\documentclass[11pt]{amsart}
\usepackage{amsfonts,latexsym,amsthm,amssymb,graphicx}
\usepackage[all]{xy}
\usepackage{hyperref}
\usepackage[usenames]{color} 

\DeclareFontFamily{OT1}{rsfs}{}
\DeclareFontShape{OT1}{rsfs}{n}{it}{<-> rsfs10}{}
\DeclareMathAlphabet{\mathscr}{OT1}{rsfs}{n}{it}

\CompileMatrices

\newtheorem{theorem}{Theorem}[section]
\newtheorem{lemma}[theorem]{Lemma}
\newtheorem{corol}[theorem]{Corollary}
\newtheorem{prop}[theorem]{Proposition}
\newtheorem{claim}[theorem]{Claim}
\newtheorem{quest}{Question}

{\theoremstyle{definition} \newtheorem{defin}[theorem]{Definition}}
{\theoremstyle{definition} \newtheorem{definthm}[theorem]{Definition/Theorem}}
{\theoremstyle{remark} \newtheorem{remark}[theorem]{Remark}
\newtheorem{example}[theorem]{Example}}

\newtheorem*{propnono}{Proposition}
\newtheorem*{conjnono}{Conjecture}
\newtheorem{thmrm}{Theorem}

\numberwithin{equation}{section}

\newcommand{\Abb}{{\mathbb{A}}}
\newcommand{\Nbb}{{\mathbb{N}}}
\newcommand{\Pbb}{{\mathbb{P}}}
\newcommand{\Qbb}{{\mathbb{Q}}}
\newcommand{\Rbb}{{\mathbb{R}}}
\newcommand{\Zbb}{{\mathbb{Z}}}

\newcommand{\cN}{{\mathscr N}}
\newcommand{\cO}{{\mathscr O}}
\newcommand{\cP}{{\mathscr P}}
\newcommand{\cK}{{\mathscr K}}
\newcommand{\cR}{{\mathscr R}}
\newcommand{\cV}{{\mathscr V}}
\newcommand{\rkR}{r}
\newcommand{\id}{\mathrm{id}}
\newcommand{\ue}{\underline e}
\newcommand{\ut}{\underline t}
\newcommand{\uu}{\underline u}
\newcommand{\uv}{\underline v}
\newcommand{\uw}{\underline w}
\newcommand{\uN}{\underline N}
\newcommand{\Til}[1]{{\widetilde{#1}}}
\newcommand{\saf}{\,}
\newcommand{\qede}{\hfill$\lrcorner$}

\DeclareMathOperator{\PGL}{PGL}
\DeclareMathOperator{\rk}{rk}
\DeclareMathOperator{\codim}{codim}
\DeclareMathOperator{\Vol}{Vol}

\title[Lorentzian polynomials, Segre classes, and adjoint polynomials]{Lorentzian polynomials, 
Segre classes, and adjoint polynomials of convex polyhedral cones\footnote{2023/12/01}
}
\author{Paolo Aluffi}
\address{
Mathematics Department, 
Florida State University,
Tallahassee FL 32306, U.S.A.
}
\email{aluffi@math.fsu.edu}

\begin{document}

\begin{abstract}
We consider polynomials expressing the cohomology classes of subvarieties of products of 
projective spaces, and limits of positive real multiples of such polynomials. We study the
relation between these {\em covolume polynomials\/} and Lorentzian polynomials.
While these are distinct notions, we prove that, like Lorentzian polynomials, covolume
polynomials have M-convex support and generalize the notion of log-concave sequences.
In fact, we prove that covolume polynomials are `sectional log-concave', that is, 
the coefficients of suitable restrictions of these polynomials form log-concave sequences.

We observe that Chern classes of globally generated bundles give rise to covolume polynomials,
and use this fact to prove that certain polynomials associated with {\em Segre classes\/} of subschemes
of products of projective spaces are covolume polynomials. We conjecture that the same polynomials
may be Lorentzian after a standard normalization operation.

Finally, we obtain a combinatorial application of a particular case of our Segre class result. 
We prove that the {\em adjoint polynomial\/} of a convex polyhedral cone contained in the 
nonnegative orthant, and sharing a face with it, is a covolume polynomial. This implies that
these adjoint polynomials are M-convex and sectional log-concave, and in fact {\em dually
Lorentzian,\/} that is, Lorentzian after a certain change of variables.
\end{abstract}

\maketitle


\section{Introduction}\label{S:intro}
This paper consists of three parts. In the first part we consider `covolume
polynomials', that is, limits of positive multiples of polynomials whose coefficients 
are multidegrees of
(irreducible) subvarieties of products of projective spaces. We are interested in 
studying covolume polynomials from the point of view of {\em Lorentzian\/} 
polynomials. Petter Br\"{a}nd\'{e}n and June Huh defined and extensively 
studied Lorentzian polynomials in~\cite{MR4172622}. An equivalent notion,
{\em completely log-concave\/} polynomials, was introduced at the same time
by Nima Anari, Kuikui Liu, Shayan Oveis Gharan and Cynthia Vinzant 
(\cite{Vinzant,MR4332671}).
Lorentzian polynomials 
generalize the notion of (ultra) log-concavity, in the sense that a homogeneous 
polynomial in two variables is Lorentzian if and only if its coefficients 
are nonnegative and form an ultra log-concave sequence with no internal zeros.
In fact, the restriction of every Lorentzian polynomial to a plane spanned by two
nonnegative vectors satisfies this ultra-log-concave property.
Br\"{a}nd\'{e}n and Huh prove that `volume polynomials' of projective varieties
are necessarily Lorentzian (\cite[Theorem~4.6]{MR4172622}). The covolume
polynomials studied in this note (or, rather, their normalizations in the sense of
Br\"{a}nd\'{e}n and Huh, see~\cite[Corollary~3.7]{MR4172622}) are a dual notion 
to volume polynomials, in the following sense: the volume polynomial of a subvariety $W$
of a product of projective spaces, with respect to the divisors defined by the hyperplane
classes of the factors, records the homology class of $W$; the corresponding covolume
polynomial similarly records the cohomology class of $W$.
While they are not necessarily Lorentzian, we prove that covolume polynomials share 
several properties of note with Lorentzian polynomials.
Some of these properties are due to their connection with volume polynomials;
others are a consequence of Olivier Debarre's beautiful extension of the
Fulton-Hansen connectedness theorem (\cite{MR1420927}).

Covolume polynomials provide an alternative generalization of log-concavity, 
which we call `sectional log-concavity' (Definition~\ref{def:slc}). Briefly, a homogeneous 
polynomial is sectional log-concave if and only if the coefficients of its restrictions to planes 
spanned by vectors with nonnegative components form a log-concave sequence with no 
internal zeros. (In the same sense, Lorentzian polynomials are `sectional
ultra-log-concave'.) Among other properties, we prove (Corollary~\ref{cor:slc}):

\begin{thmrm}\label{thm:covsec}
Covolume polynomials are sectional log-concave.
\end{thmrm}

In particular, the normalization of a covolume polynomial in two variables is 
necessarily Lorentzian. In higher dimension, the relation between covolume 
polynomials and Lorentzian polynomials is subtler. Examples show that covolume 
polynomials are not necessarily Lorentzian, even after normalization
(Example~\ref{ex:notlor}), and Lorentzian polynomials are not necessarily covolume
polynomials
(Example~\ref{ex:notcov}). With each covolume polynomial there is associated a
Lorentzian polynomial (Proposition~\ref{prop:covtolor}); in fact covolume polynomials
are {\em dually Lorentzian\/} in the sense of~\cite{dualLor}. 
It follows that the support 
of a covolume polynomial is M-convex in the sense of discrete convex 
analysis~\cite{MR1997998}, i.e., a `polymatroid' (Corollary~\ref{cor:Mconv}).
In analogy with the behavior of Lorentzian polynomials, we prove 
that if $f(\ut)$ is a covolume polynomial, then so is $f(A\ut)$ for any matrix $A$ with 
nonnegative entries (Theorem~\ref{thm:nonnegcc}). Sectional log-concavity is
a corollary of this result, which has other convenient consequences, such as
the fact that the product of two covolume polynomials is a covolume polynomials.
The set of Lorentzian polynomials is similarly preserved by nonnegative changes 
of variables by~\cite[Theorem~2.10]{MR4172622}. By~\cite[Theorem~5.12]{dualLor}, 
so is the set of dually Lorentzian polynomials.\smallskip

The second part of the paper deals with {\em Segre zeta functions\/} of closed
subschemes of products of projective spaces. Segre classes are a key ingredient
in Fulton-MacPherson intersection theory. In previous work (\cite{MR3709134})
we constructed a univariate rational function~$\zeta_I(t)$ encoding the information
of the Segre class of the subscheme defined in projective spaces of arbitrarily large 
dimension by a set $I$ of homogeneous polynomials. In this paper we
extend the construction to subschemes of products of projective spaces; if $Z$ is a 
closed subscheme of $\Pbb^{n_1}\times \cdots \times \Pbb^{n_\ell}$
defined by a set $I$ of multihomogeneous polynomials, the corresponding 
`Segre zeta function' $\zeta_I(t_1,\dots, t_\ell)$ of $Z$ may be 
defined as a power series in $\ell$ variables.
This power series only depends on the ideal generated by $I$,
and on the subscheme defined by $I$ for $n_i\gg 0$, so we may refer to the 
corresponding series $\zeta_I$ as the Segre zeta function of an ideal or of a scheme;
but it is useful to keep track of a specific finite generating set $I$.
We point out (Theorem~\ref{thm:szf2})
that it $\zeta_I$ is  a rational function, with poles controlled by the multidegrees of 
elements of $I$. Thus, once a generating set $I$ is chosen, the interesting 
part of the information of a Segre zeta function is its numerator; we will simply refer to this
as the `numerator of $\zeta_I$'.
Constraints on the possible numerators of Segre zeta functions yield nontrivial information 
on Segre classes, and this is our motivation in studying them. We note that the case of
products of projective spaces is `universal' in certain contexts: for example, one can 
associate a Segre zeta function to {\em monomial\/} schemes in arbitrary varieties,
but every such power series is the Segre zeta function of a monomial scheme in a product
of projective spaces. 
(Indeed, the Segre zeta function only depends on the Newton region associated with the
monomial subscheme, by~\cite[Theorem~1.1]{MR3576538}.)
Thus, the situation considered in this paper is less restrictive than may appear.
Our main result in this part is:

\begin{thmrm}\label{thm:mainII}
With notation as above, the homogenization of the numerator of $1-\zeta_I$ is a 
covolume polynomial. If the projective normal cone of $Z$ is irreducible, then the 
homogenization of the numerator of $\zeta_I$ is a covolume polynomial.
\end{thmrm}

See Theorems~\ref{thm:R} and~\ref{thm:P} for more complete statements. 
This implies that e.g., the numerator of $1-\zeta_I$ is sectional log-concave
and M-convex in the natural sense. In particular, the coefficients of the numerator
of $1-\zeta_I(t)$ in the univariate case (that is, the case of subschemes of $\Pbb^n$
studied in~\cite{MR3709134}) necessarily form a log-concave sequence of nonnegative
integers with no internal zeros. 

Analogous consequences hold for the numerator of $\zeta_I$ under an irreducibility
hypothesis on the normal cone of the subscheme defined by $I$. Some such 
hypothesis is necessary: examples show 
that the numerator of the Segre zeta function is {\em not\/} M-convex or sectional 
log-concave in general (Example~\ref{ex:Pnotcov}), implying that its homogenization 
is not necessarily a covolume or Lorentzian polynomial. Whether the irreducibility 
hypothesis stated in Theorem~\ref{thm:mainII} can be
significantly weakened is an interesting question. Brandon Story (\cite{story}) has 
verified that in the univariate case, i.e., for $\ell=1$, the coefficients of the numerator 
of the Segre zeta function form a log-concave sequence without internal zeros for 
several families of subschemes $Z\subseteq \Pbb^{n_1}$, including many cases of
reducible subschemes.

The situation with $1-\zeta_I$, whose numerator is a covolume polynomial by
Theorem~\ref{thm:mainII}, is also intriguing. Experimental evidence suggests the 
following.

\begin{conjnono}
Let $Z$ be a closed subscheme of $\Pbb^{n_1}\times \cdots \times \Pbb^{n_\ell}$,
and $I$ a generating set for its ideal.
Then the normalization of the homogenization of the numerator of 
$1-\zeta_I(t_0,\dots, t_\ell)$ is Lorentzian.
\end{conjnono}

Since there are covolume polynomials whose normalization is not Lorentzian, this 
conjecture does not follow formally from Theorem~\ref{thm:mainII}; if true, it appears 
to be a novel and unexpected phenomenon. The reader could compare this 
conjecture with Conjecture~15 in~\cite{MR4419063}, to the effect that Schubert 
polynomials are expected to have Lorentzian normalization.  Schubert polynomials 
also are covolume polynomials (Example~\ref{ex:Schub}); this fact alone does 
not explain why their normalizations should be Lorentzian. 

This may be a good place to quote Karim Adiprasito, June Huh, and Eric Katz
(\cite{MR3586249}): 
{\em ``We believe that behind any log-concave sequence that appears in nature there is\dots 
a `Hodge structure' responsible for the log-concavity.''\/}
We wonder whether the sectional log-concavity of the numerator of $1-\zeta_I$ may
signal the presence of a new structure in the sense meant in this reference. These 
polynomials (and their normalizations) are not directly expressed as volume polynomials 
of projective varieties.

Theorem~\ref{thm:mainII} follows from a general statement that may have different 
applications. Let $X$ be an irreducible variety, 
$q: X \to \Pbb^{n_1}\times \cdots \times \Pbb^{n_\ell}$ a proper map,
and let $\cR$ be a globally generated vector bundle on $X$. We can write 
\[
q_*(c(\cR)\cap [X]) = \sum_{0\le i_j\le n_j} a_{i_1\dots i_\ell} h_1^{i_1}\cdots h_\ell^{i_\ell}
\cap [\Pbb^{n_1}\times\cdots \times \Pbb^{n_\ell}]\saf,
\]
where $h_j$ is the pull-back of the hyperplane class from the $j$-th factor of the product.
Let
\[
P(t_1,\dots, t_\ell)=\sum_{i_k\le s_k} a_{i_1\dots i_\ell} t_1^{i_1}\cdots t_\ell^{i_\ell}\saf.
\]
\begin{propnono}
The homogenization of $P(t_1,\dots, t_\ell)$ is a covolume polynomial.
\end{propnono}
(See Proposition~\ref{prop:irremap}.)
In particular, this polynomial must be M-convex and sectional log-concave. For example, 
the degrees of the Chern classes of a globally generated vector bundle over projective 
space must form a log-concave sequence with no internal zeros. For remarks 
implying the same conclusion (independently), see~\cite[\S1.3(B), \S9]{MR4659531}, drawing 
on~\cite[\S1.6]{MR2095471}.\smallskip

While our main motivation in establishing Theorem~\ref{thm:mainII} is intrinsic to the
theory of Segre classes, we offer a concrete combinatorial application of this result
in the particular case of {\em monomial\/} schemes:
in the third part of the paper we consider {\em adjoint polynomials\/} of convex 
polytopes (or, equivalently, convex polyhedral cones). These polynomials were introduced
by Joe Warren (\cite{MR1431788}); the case of polygons had been considered earlier
by Eugene Wachspress (\cite{MR0426460}). Adjoint polynomials are used in the definition
of ``Wachspress coordinates'' of convex polytopes; we refer the reader to~\cite{MR4156995}
for a discussion of the context and recent progress in the study of Wachspress coordinates.
We also note that adjoint polynomials appear as numerators of canonical forms of certain
`positive geometries' introduced in the study of scattering amplitudes. This is nicely explained
in unpublished notes by Christian Gaetz~\cite{gaetz}; see~\cite{AaCFoP} for recent work on
adjoints from this perspective.

The adjoint curve of a polygon is the curve of minimal degree containing the points of
intersection of pairs of lines extending non-adjacent edges of the polygon:
\begin{center}
\includegraphics[scale=.3]{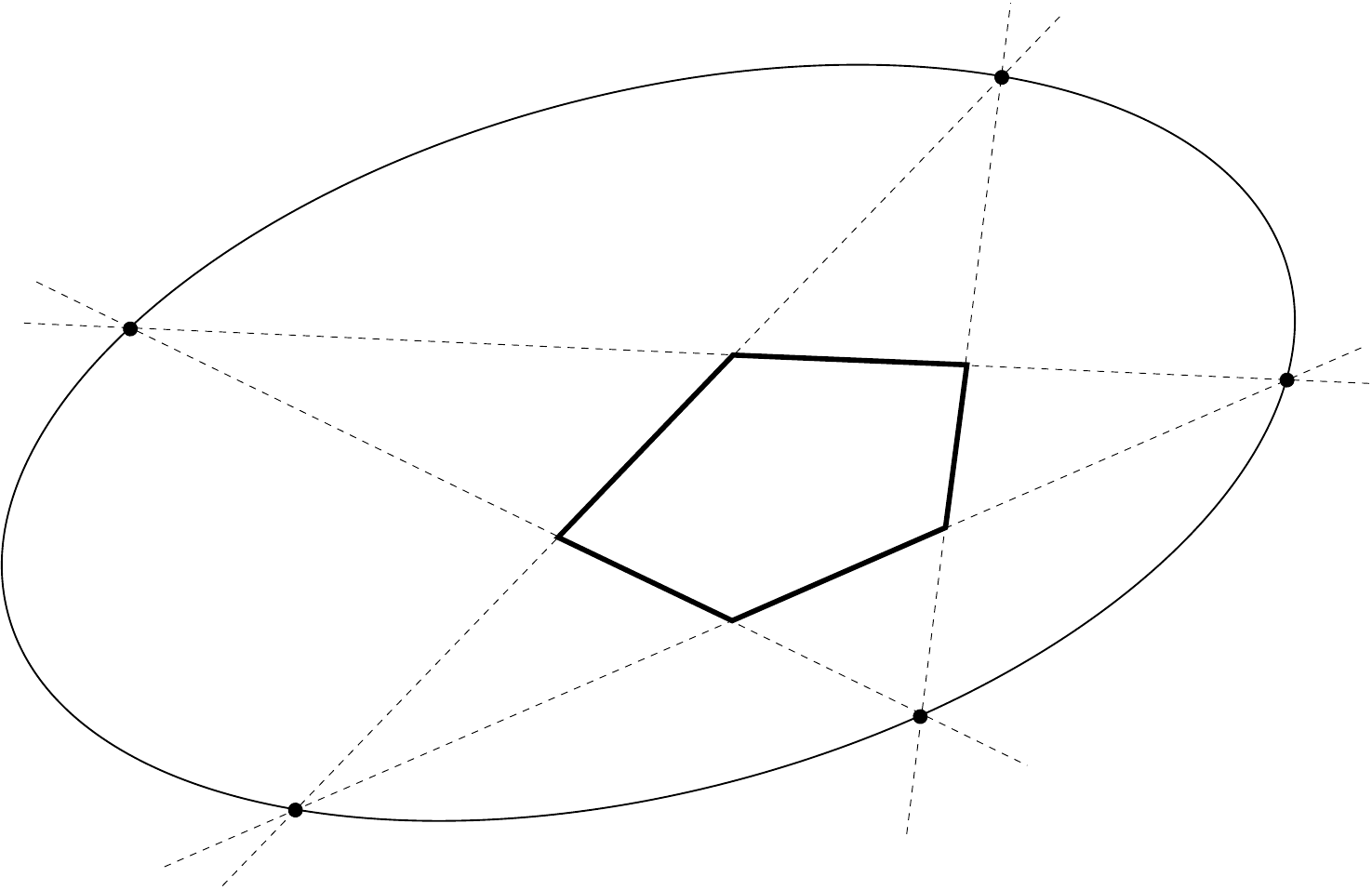}
\end{center}
an analogous description holds in any dimension. In~\cite{MR4156995}, Kathl\'en Kohn and 
Kristian Ranestad observe that the numerator of the Segre zeta function of a monomial
scheme may be interpreted as an adjoint polynomial of a related polytope. We sharpen
this result in Proposition~\ref{prop:adjasS}, by proving that the adjoint polynomial of every
convex polyhedral cone contained in the nonnegative orthant and containing a face of the
nonnegative orthant is the numerator of $1-\zeta_I$ for a suitable choice of $I$. 
Together with Theorem~\ref{thm:mainII}, this implies the following.

\begin{thmrm}\label{thm:adpo}
Adjoint polynomials of convex polyhedral cones contained in the nonnegative orthant and
sharing a face with it are covolume polynomials.
\end{thmrm}

See Theorem~\ref{thm:adjpol} for a more complete statement.
In particular, these adjoint polynomials are necessarily M-convex and sectional log-concave
(and in fact dually Lorentzian in the sense of~\cite{dualLor}).

It is conceivable that this result may extend to all convex polyhedral cones contained in
the nonnegative orthant. We prove that Theorem~\ref{thm:adpo} does extend to a certain 
class of cones which we call `orthantal'; see Definition~\ref{def:orthantal} and 
Corollary~\ref{cor:orthan}.

It is also conceivable that adjoint polynomials of all convex polyhedral cones contained in the 
nonnegative orthant are actually Lorentzian after normalization. The conjecture stated above
would imply this fact for orthantal cones. 

\smallskip 

{\em Acknowledgments.} This work
was supported in part by the Simons Foundation, collaboration grant~\#625561,
and by an FSU `COFRS' award.
The author thanks Caltech for the hospitality. The author also thanks Petter Br\"{a}nd\'{e}n, 
Kathl\'en Kohn, Matt Larson, Jonathan Monta\~no, and the referee for useful comments.


\section{Covolume polynomials}\label{partI}
We work over an algebraically closed field $k$. {\em Varieties\/} are assumed to be irreducible,
but as their irreducibility is key to the main objects studied in this paper, we occasionally remind 
the reader of this fact.

We consider (irreducible) subvarieties of products of projective spaces 
$\Pbb:=\Pbb^{n_0}\times \cdots\times
\Pbb^{n_\ell}$ and their classes in the Chow ring of $\Pbb$. We will let $h_j$ denote the pull-back 
of the hyperplane class from the $j$-th factor. Every class in the codimension-$d$
graded piece $A^d(\Pbb)$ of the Chow ring of $\Pbb$ may be written uniquely as
\[
\alpha=\sum_{\sum i_j=d} a_{i_0\dots i_\ell}\, h_0^{i_0}\cdots h_\ell^{i_\ell}\cap [\Pbb]\saf,
\]
where $0\le i_j\le n_j$ and $a_{i_0\dots i_\ell}\in \Zbb$. For any choice of
indeterminates $t_0,\dots, t_\ell$, we associate with~$\alpha$ the polynomial
\[
P_\alpha(t_0,\dots, t_\ell):=
\sum_{0\le i_j, \sum i_j=d} a_{i_0\dots i_\ell}\, t_0^{i_0}\cdots t_\ell^{i_\ell}\saf.
\]
In other words, $P_\alpha$ is the (unique) polynomial with integer coefficients and of degree 
$\le n_j$ in the $j$-th variable such that
\[
\alpha=P_\alpha(h_0,\dots, h_\ell)\in A(\Pbb)\cong 
\Zbb[h_0,\dots, h_\ell]/(h_0^{n_0+1},\dots, h_\ell^{n_\ell+1})\saf.
\]

\begin{defin}\label{def:covolume}
A polynomial $P$ with nonnegative real coefficients is a {\em covolume\/} polynomial if it is
a limit of polynomials of the form $c P_{[W]}$ for a positive real number $c$ and a closed 
subvariety~$W$ of a product of projective spaces.
\qede\end{defin}

\begin{remark}\label{rem:event}
If $W$ is a subvariety of $\Pbb:=\Pbb^{n_0}\times \cdots\times \Pbb^{n_\ell}$
with class $[W]=P(h_0,\dots, h_\ell)\cap [\Pbb]$, then for all $m_j\ge n_j$, the `cone' $W'$
over $W$ in $\Pbb':=\Pbb^{m_0}\times \cdots\times \Pbb^{m_\ell}$ is a subvariety
with class $[W']=P(h'_0,\dots, h'_\ell)\cap [\Pbb']$, where $h'_j$ denotes the pull-back of the
hyperplane class from the $j$-th factor of $\Pbb'$.

Thus, if $P=P_{[W]}$ with $W\subseteq 
\Pbb^{n_0}\times \cdots\times \Pbb^{n_\ell}$, then
we may in fact assume all $n_j\gg 0$, or even $n_0=\cdots=n_{\ell}\gg 0$.
\qede\end{remark}

\begin{example}
Constant polynomials (including $0$) trivially are covolume polynomials.
Linear polynomials with nonnegative real coefficients are covolume polynomials.
Indeed, by continuity we can assume that the coefficients are positive and rational, and 
up to a multiple we may then assume the coefficients to be integers. For positive
integers $a_0,\dots, a_\ell$, general sections of $\cO(a_0h_0+\cdots+a_\ell h_\ell)$ 
on $(\Pbb^n)^{\ell+1}$ are irreducible for $n\ge 2$ and any $\ell\ge 0$. The corresponding
polynomial is $a_0 t_0+\cdots+a_\ell t_\ell$.
\qede\end{example}

\begin{example}\label{ex:Schub}
For a nontrivial example, we note that the Schubert polynomial 
$\mathfrak S_w(t_0,\dots, t_\ell)$ associated with a permutation $w\in \mathcal S_{\ell+1}$ 
is a covolume polynomial. This follows from the proof of~\cite[Theorem~6]{MR4419063}.
\qede\end{example}

We are interested in log-concavity properties of covolume polynomials.  
We recall that a sequence $a_0,\dots,a_n$ of nonnegative real numbers is {\em log-concave\/}
if $\forall i$, $a_i^2\ge a_{i-1}a_{i+1}$. Further, the sequence `has no internal zeros'
if $\forall i\le j\le k$, $a_i a_k\ne 0 \implies a_j\ne 0$. We say that a polynomial is 
log-concave if its coefficients form a log-concave sequence with no internal
zeros. Powerful generalizations of this notion to polynomials in more variables
have been considered in the literature: among these {\em strongly\/} log-concave
polynomials (\cite{MR2683227}), {\em completely\/} log-concave polynomials (\cite{Vinzant}),
{\em Lorentzian\/} polynomials (\cite{MR4172622}). These notions are equivalent for
homogeneous polynomials, as proved in \cite[Theorem~2.30]{MR4172622}.
We will consistently refer to this property by the term {\em Lorentzian,\/} and
use~\cite{MR4172622} as a reference for facts concerning this notion.
According to~\cite{MR4172622}, a degree $d$ homogeneous polynomial $F$ with real
coefficients is {\em strictly Lorentzian\/} if all of its coefficients are positive and all
degree-$2$ partial derivatives of $F$ have nonsingular Hessian with exactly one
positive eigenvalue. (This condition is vacuously true if the degree of $F$ is $0$ or 
$1$.) The polynomial is {\em Lorentzian\/} if it is a limit of strictly Lorentzian polynomials. 
In~\cite[\S2.4]{MR4172622}, Br\"{a}nd\'{e}n and Huh prove that $F$ is Lorentzian 
if and only if its degree-$2$ partials have at most one positive eigenvalue and the 
support of $F$ is M-convex, that is, it satisfies an exchange property that will be 
recalled below.

We will establish that covolume polynomials share certain key properties with Lorentzian 
polynomials, and may also be viewed as a generalization of log-concave polynomials. 
Again, we say that a homogeneous bivariate polynomial
\[
P(u,v):=\sum_{i+j=d} a_{ij} u^i v^j\in \Rbb[u,v]
\]
is {\em log-concave\/} if and only if its coefficients $a_{d0},\dots, a_{0d}$ form a 
nonzero log-concave sequence of nonnegative real numbers with no internal zeros. 

\begin{defin}\label{def:slc}
For $\ell\ge 1$, a homogeneous polynomial $P(\ut)\in \Rbb[t_0,\dots, t_\ell]$ is 
{\em sectional log-concave\/} if for all $(\ell+1)\times 2$ matrices $A$ with 
nonnegative real entries, the polynomial $P(A\binom uv)$ is log-concave or 
identically~$0$.
\qede\end{defin}

A bivariate homogeneous polynomial is sectional log-concave if it is log-concave 
(this is not obvious, cf.~Remark~\ref{rem:lcslc}),
and in general a homogeneous polynomial is sectional log-concave if all its 
restrictions to `nonnegative' lower 
dimensional subspaces are sectional log-concave. 
Lorentzian polynomials are sectional log-concave;
in fact, if $P(\ut)$ is Lorentzian and $A$ is as above, then $P(A\binom uv)$ is 
{\em ultra-\/}log-concave as a consequence of 
\cite[Theorem~2.10 and Example~2.26]{MR4172622}.
(We recall that a homogeneous polynomial $\sum_{i+j=d}a_{ij}u^iv^j$ is 
`ultra'-log-concave if the sequence $i!j!a_{ij}$ is log-concave with no internal
zeros. This is a stronger condition than log-concavity.)
One of our main goals in this section is to prove that
covolume polynomials are sectional log-concave.

The following characterization of two-variable covolume polynomials is a simple corollary 
of another result of June Huh (\cite[Theorem~21]{MR2904577}). Following~\cite{MR4172622},
we consider the {\em normalization\/} operator on the polynomial ring, defined on monomials 
by $N(t_0^{i_0}\cdots t_\ell^{i_\ell}):=t_0^{i_0}\cdots t_\ell^{i_\ell}/i_0!\cdots i_\ell!$.
By~\cite[Corollary~3.7]{MR4172622}, the normalization of a Lorentzian polynomial
is Lorentzian.

\begin{lemma}\label{lem:twovac}
A nonzero homogeneous bivariate polynomial
\[
P(u,v):=\sum_{i+j=d} a_{ij} u^i v^j\in \Rbb[u,v]
\]
is a covolume polynomial if and only if it is log-concave, that is, its coefficients 
$a_{d0},\dots, a_{0d}$ form a nonzero log-concave sequence of nonnegative real 
numbers with no internal zeros. 

Therefore, a homogeneous polynomial $P(u,v)\in \Rbb[u,v]$ is a covolume polynomial if
and only if its normalization is Lorentzian.
\end{lemma}

\begin{proof}
Assume that $P(u,v)\in \Rbb[u,v]$ is a covolume polynomial. By continuity, it suffices to 
verify the statement when the coefficients of $P$ are rational. Therefore, we may assume 
that there exists a positive rational $c$ and an irreducible subvariety $W$ of codimension~$d$ 
in $\Pbb^n\times\Pbb^n$, with $n\gg 0$ (cf.~Remark~\ref{rem:event}), such that
\[
[W]=\sum_{i+j=d} ca_{ij} h^i k^j \cap [\Pbb^n\times \Pbb^n]\saf.
\]
Here $h$, resp.~$k$ denote the pull-backs of the hyperplane classes from the first, resp.~second
factor. Thus,
\[
[W]=\sum_{i+j=d} ca_{ij} [\Pbb^{n-i}\times \Pbb^{n-j}]
=\sum_{i+j=\dim W} ca_{n-i,n-j} [\Pbb^i \times \Pbb^j]
\]
is the class of an irreducible subvariety of $\Pbb^n\times \Pbb^n$.
By~\cite[Theorem~21]{MR2904577}, the coefficients $ca_{n-i,n-j}$ form a nonzero log-concave
sequence with no internal zeros, and it follows that the same holds for $a_{d0},\dots, a_{0d}$.

Conversely, assume that $a_{d0},\dots, a_{0d}$ form a nonzero log-concave sequence of 
nonnegative real numbers with no internal zeros. Such sequences are limits of 
log-concave sequences of nonnegative rational numbers with no internal zeros,
so we may assume $a_{d0},\dots, a_{0d}$ are rational.
Clearing denominators, there exists $c\in \Qbb_{>0}$ such
that the coefficients $ca_{ij}$ are integers; so $ca_{d0},\dots, ca_{0d}$ form a log-concave
sequence of nonnegative integers with no internal zeros. Again by~\cite[Theorem~21]{MR2904577}, 
a positive integer multiple of 
\[
\sum_{i+j=d} ca_{ij} [\Pbb^j \times \Pbb^i]
=\sum_{i+j=d} ca_{ij} h^i k^j \cap [\Pbb^d\times \Pbb^d]
\]
is the class of an irreducible subvariety of $\Pbb^d \times \Pbb^d$, and it follows that
$P(u,v)$ is a covolume polynomial.

The coefficients of the polynomial $\sum_{i+j=d} a_{ij} u^i v^j$ form a log-concave sequence
with no internal zeros if and only if the coefficients of the normalization
$\sum_{i+j=d} a_{ij} u^i v^j/i! j!$ form an {\em ultra\/} log-concave sequence with no internal
zeros, if and only if the normalization is Lorentzian, cf.~\cite[Example~2.26]{MR4172622}.
\end{proof}

Lemma~\ref{lem:twovac} raises the natural question of whether the normalization of a
covolume polynomial may be Lorentzian in general. This is not the case.

\begin{example}\label{ex:notlor}
Let $W$ be the image of the embedding
\[
\Pbb^1\times \Pbb^1\times \Pbb^1 \hookrightarrow \Pbb^7\times \Pbb^3\times \Pbb^1
\]
obtained from the Segre embedding $\Pbb^1\times \Pbb^1\times \Pbb^1 \hookrightarrow 
\Pbb^7$, the Segre embedding of the first two factors $\Pbb^1\times \Pbb^1 \hookrightarrow \Pbb^3$, 
and the identity $\Pbb^1 \to \Pbb^1$ on the third factor. Let 
\[
[W]=\sum_{i_0+i_1+i_2=8} a_{i_0 i_1 i_2} h_0^{i_0} h_1^{i_1} h_2^{i_2}\cap
[\Pbb^7\times \Pbb^3\times \Pbb^1]\saf.
\]
Denoting by $k_i$ the pull-back of the hyperplane classes from the $i$-th factor
of the product $\Pbb^1\times \Pbb^1\times\Pbb^1$, we have
\begin{align*}
a_{i_0 i_1 i_2} &= \int h_0^{7-i_0} h_1^{3-i_1} h_2^{1-i_2}\cap [W] \\
&=\text{coeff.~of $k_0 k_1 k_2$ in } (k_0+k_1+k_2)^{7-i_0} (k_0+k_1)^{3-i_1} k_2^{1-i_2}\saf.
\end{align*}
It follows that
\[
P_{[W]}(t_0,t_1,t_2)=
2t_0^7t_1 + 2t_0^6t_1^2 + 2t_0^6t_1t_2 + 2t_0^5t_1^3 + 4t_0^5t_1^2t_2 + 
6t_0^4t_1^3t_2
\]
is a covolume polynomial. The normalization of $P_{[W]}$ is
\[
N(P_{[W]})=\frac 1{2520}t_0^7t_1 + \frac 1{720}t_0^6t_1^2 + \frac 1{360}t_0^6t_1t_2 + 
\frac 1{360}t_0^5t_1^3 + \frac 1{60}t_0^5t_1^2t_2 + \frac 1{24}t_0^4t_1^3t_2\saf,
\]
and this polynomial is {\em not\/} Lorentzian. Indeed,
\[
\frac{\partial^5}{\partial t_0^5}\frac{\partial}{\partial t_1} N(P_{[W]})
=t_0^2 + 2t_0 t_1 + 2 t_0 t_2 + t_1^2 + 4 t_1 t_2\saf,
\]
with Hessian
\[
\begin{pmatrix}
2 & 2 & 2 \\
2 & 2 & 4 \\
2 & 4 & 0
\end{pmatrix}\saf.
\]
This matrix has {\em two\/} positive eigenvalues, contrary to the requirement for Lorentzianity
(cf.~\cite[\S2.4]{MR4172622}).
\qede\end{example}

On the other hand, the normalization of a simple transformation of a covolume polynomial 
{\em is\/} Lorentzian. 

\begin{prop}\label{prop:covtolor}
Let $P(t_0,\dots,t_\ell)$ be a covolume polynomial. Let $n_j$, $j=0,\dots, \ell$, be
any integers such that
\[
Q(u_0,\dots, u_\ell) = u_0^{n_0}\cdots u_\ell^{n_\ell} P\left(\frac 1{u_0},\dots, \frac 1{u_\ell}
\right)
\]
is a polynomial. Then $N(Q)$ is Lorentzian. 
\end{prop}

In fact, we will prove that, possibly up to inessential factors, 
$N(Q)$ is a limit of {\em volume 
polynomials\/} in the sense of~\cite[\S4.2]{MR4172622}, that is, polynomials of the form
\[
\int (u_0 h_0+\cdots +u_\ell h_\ell)^{\dim W}\cap [W]
\]
where $W$ is an irreducible variety and $h_0,\dots, h_\ell$ is a collection of nef divisors 
on $W$.

\begin{remark}
Proposition~\ref{prop:covtolor} shows that covolume polynomials are {\em dually
Lorentzian.\/} This notion is introduced and thoroughly studied in~\cite{dualLor}.
\qede\end{remark}

\begin{proof}
For all $j$, $N(Q)$ is Lorentzian if $N(u_j Q)$ is Lorentzian; indeed, the former is the 
derivative of the latter with respect to $u_j$, and derivatives of Lorentzian polynomials
are Lorentzian by definition. Therefore, it suffices to show that the normalization of
\[
Q(u_0,\dots, u_\ell) = u_0^d\cdots u_\ell^d P\left(\frac 1{u_0},\dots, \frac 1{u_\ell}
\right)
\]
is Lorentzian, where $d=\deg P$. By continuity and up to a constant multiple we may then
assume that $P=P_{[W]}$ is the covolume polynomial
associated with the class of an irreducible subvariety $W$ of a product of projective spaces;
and by the same argument used above, it suffices to show that there exist $n_0\ge d,\dots, 
n_\ell\ge d$ such that 
\[
Q_{[W]}(u_0,\dots, u_\ell) = u_0^{n_0}\cdots u_\ell^{n_\ell} P_{[W]}\left(\frac 1{u_0},\dots, \frac 1{u_\ell}
\right)
\]
has Lorentzian normalization. 
By Remark~\ref{rem:event}, we may choose 
$W\subseteq \Pbb^{n_0}\times \cdots\times \Pbb^{n_\ell}$
with $n_j\gg 0$. Then
\[
P_{[W]}=\sum_{i_0+\cdots+i_\ell=d} a_{i_0\dots i_\ell} t_0^{i_0}\cdots t_\ell^{i_\ell}
\]
where $0\le i_j\le n_j$ for all $j$ and
\[
[W]=\sum_{i_0+\cdots+i_\ell=d} a_{i_0\dots i_\ell} h_0^{i_0}\cdots h_\ell^{i_\ell} 
\cap[\Pbb^{n_0} \times \cdots \times \Pbb^{n_\ell}]
\]
with the usual notation. We have
\[
a_{i_0\dots i_\ell} = \int h_0^{n_0-i_0}\cdots h_\ell^{n_\ell-i_\ell}\cap [W]\saf.
\]
Therefore
\begin{align*}
N(Q_{[W]}(u_0,\dots, u_\ell)) &= N\left(u_0^{n_0}\cdots u_\ell^{n_\ell} 
P_{[W]}\left(\frac 1{u_0},\dots, \frac 1{u_\ell}
\right)\right)\\
&= \sum_{i_0+\cdots+i_\ell=d} a_{i_0\dots i_\ell} \frac{u_0^{n_0-i_0}}{(n_0-i_0)!}
\cdots \frac{u_\ell^{n_\ell-i_\ell}}{(n_\ell -i_\ell)!} \\
&= \sum_{j_0+\cdots+j_\ell=\dim W} a_{n_0-j_0\dots n_\ell-j_\ell} \frac{u_0^{j_0}}{j_0!}
\cdots \frac{u_\ell^{j_\ell}}{j_\ell!} \\
&= \sum_{j_0+\cdots+j_\ell=\dim W} 
\left(\int h_0^{j_0}\cdots h_\ell^{j_\ell}\cap [W]\right) \frac{u_0^{j_0}}{j_0!}
\cdots \frac{u_\ell^{j_\ell}}{j_\ell!} \\
&=\frac 1{(\dim W)!} \int \sum_{j_0+\cdots+j_\ell=\dim W}\binom{\dim W}{j_0\dots j_\ell}
(u_0 h_0)^{j_0}\cdots (u_\ell h_\ell)^{j_\ell}\cap [W] \\
&=\frac 1{(\dim W)!} \int (u_0 h_0+\cdots +u_\ell h_\ell)^{\dim W}\cap [W]\saf.
\end{align*}
This shows that, up to a scalar factor, $N(Q_{[W]})$ is a volume polynomial. It follows that 
$N(Q_{[W]})$ is Lorentzian by~\cite[Theorem~4.6]{MR4172622}, and this concludes the
argument.
\end{proof}

\begin{remark}
A sharp version of this argument is given for the case of Schubert polynomials (cf.~Example~\ref{ex:Schub}) in~\cite[Theorem~6]{MR4419063}.
\qede\end{remark}

\begin{example}\label{ex:notcov}
As we verified in Example~\ref{ex:notlor}, the polynomial
\[
P(t_0,t_1,t_2)=
2t_0^7t_1 + 2t_0^6t_1^2 + 2t_0^6t_1t_2 + 2t_0^5t_1^3 + 4t_0^5t_1^2t_2 + 
6t_0^4t_1^3t_2
\]
is a covolume polynomial. For $n_0=7, n_1=3, n_2=1$, we have
\begin{align*}
Q(u_0,u_1,u_2) &=u_0^7 u_1^3 u_2 P\left(\frac 1{u_0},\frac 1{u_1},\frac 1{u_2}\right) \\
&= 6u_0^3 + 4u_0^2 u_1 + 2u_0^2 u_2 + 2 u_0 u_1^2 + 2 u_0 u_1 u_2 + 2 u_1^2 u_2\saf.
\end{align*}
The normalization of this polynomial, 
\[
N(Q)=u_0^3 + 2u_0^2 u_1 + u_0^2 u_2 + u_0 u_1^2 + 2 u_0 u_1 u_2 + u_1^2 u_2\saf,
\]
is Lorentzian as prescribed by Proposition~\ref{prop:covtolor}. 

This also shows that a polynomial whose normalization is Lorentzian is not necessarily
a covolume polynomial. Indeed, if $Q$ were a covolume polynomial, then by
Proposition~\ref{prop:covtolor} it would follow that $N(P)$ is Lorentzian, and we have
verified that this is not the case in Example~\ref{ex:notlor}.

Proposition~\ref{prop:covtolor} also implies that {\em Lorentzian\/} polynomials are not 
necessarily covolume polynomials. Indeed, the polynomial
\[
A(u_0,u_1,u_2,u_3)=u_0^2 u_1 + u_0^2 u_2 + u_0^2 u_3 + u_0 u_1 u_2 + u_0 u_1 u_3 
+ 4 u_0 u_2 u_3 + u_1 u_2 u_3
\]
is Lorentzian, but the normalization of
\[
t_0^2 t_1 t_2 t_3\cdot A\left(\frac 1{t_0}, \frac 1{t_1} , \frac 1{t_2}, \frac 1{t_3}\right) 
=t_2 t_3+t_1 t_3+t_1 t_2+t_0 t_3+t_0t_2+4 t_0 t_1+t_0^2
\]
is not Lorentzian: its Hessian
\[
\begin{pmatrix}
2 & 4 & 1 & 1\\
4 & 0 & 1 & 1\\
1 & 1 & 0 & 1\\
1 & 1 & 1 & 0
\end{pmatrix}
\]
has two positive eigenvalues.
\qede\end{example}

By Proposition~\ref{prop:covtolor}, every covolume polynomial may be expressed in terms 
of a Lorentzian polynomial: if $P(t_0,\dots, t_\ell)$ is a covolume polynomial, 
then there exist nonnegative integers $n_0,\dots, n_\ell$ such that
\begin{equation}\label{eq:covlor}
P(t_0,\dots, t_\ell) = t_0^{n_0}\cdots t_\ell^{n_\ell} Q\left(\frac 1{t_0},\dots, \frac 1{t_\ell}
\right)
\end{equation}
where $Q$ is a polynomial whose normalization is Lorentzian. The result that follows is a 
consequence of this observation. 

Recall that a subset $S\subseteq\Nbb^{\ell+1}$ is `M-convex'
if for all $i$ and all $\alpha,\beta\in S$ such that $\alpha_i>\beta_i$, there exists $j$ such that
$\alpha_j<\beta_j$ and $\alpha-e_i+e_j\in S$, $\beta-e_j+e_i\in S$, where $e_i$ is that
$i$-th standard unit vector. (See~\cite{MR1997998}, \cite[\S2]{MR4172622}.)
This (symmetric) {\em exchange property\/} generalizes the exchange property defining 
matroids; M-convex sets are {\em generalized polymatroids.\/}

\begin{corol}\label{cor:Mconv}
The support of a covolume polynomial is an $M$-convex set.
\end{corol}

\begin{proof}
Let 
\[
P=\sum_{i_0+\cdots+i_\ell=d} a_{i_0\dots i_\ell} t_0^{i_0}\cdots t_\ell^{i_\ell}
\]
be a covolume polynomial. The corresponding polynomial $Q$ as in 
Proposition~\ref{prop:covtolor} or~\eqref{eq:covlor} is given by
\[
Q=\sum_{j_0+\cdots+j_\ell=\sum_k n_k-d} a_{n_0-j_0\dots n_\ell-j_\ell} 
u_0^{j_0}\cdots u_\ell^{j_\ell}\saf.
\]
As the normalization of this polynomial is Lorentzian, its support, that is, the set
\[
\{(j_0,\dots, j_\ell)\,|\, a_{n_0-j_0\dots n_\ell-j_\ell}\ne 0\}
\]
is M-convex (\cite[Definition~2.6]{MR4172622}). It is then straightforward to check that 
the support of $P$, that is,
\[
\{(i_0,\dots, i_\ell)\,|\, a_{i_0\dots i_\ell}\ne 0\}
\]
also satisfies the symmetric exchange property and is therefore M-convex.
\end{proof}

Lemma~\ref{lem:twovac} and Corollary~\ref{cor:Mconv} indicate that while covolume 
polynomials (or even their normalizations) are not necessarily Lorentzian, they share 
certain key properties with Lorentzian polynomials. The result that follows is possibly
the most useful such analogue.
Recall from~\cite[Theorem~2.10]{MR4172622} that if $f(\uw)$ is a Lorentzian polynomial,
then so is $f(A\uv)$ for any matrix $A$ with nonnegative entries. The set of covolume 
polynomials is similarly preserved by nonnegative coordinate changes.

\begin{theorem}\label{thm:nonnegcc}
Let $P(\ut)\in \Rbb[t_0,\dots, t_\ell]$ be a covolume polynomial, and let $A$ be an 
$(\ell+1)\times (m+1)$ matrix with nonnegative real entries. Then 
$P(A\uu)\in \Rbb[u_0,\dots, u_m]$ is a covolume polynomial.
\end{theorem}

For example, the `dilation' (replacing a variable by a nonnegative constant multiple of the same 
variable) and `diagonalization' (setting two variables equal to each other) operators preserve 
the property of being a covolume polynomial. The following consequence is the analogue 
of~\cite[Corollary~2.32]{MR4172622}.

\begin{corol}
The product of two covolume polynomials is a covolume polynomial.
\end{corol}

\begin{proof}
By continuity and multiplication by scalar multiples, we are reduced to showing that 
$P_{[W']} (\ut) P_{[W'']}(\ut)$ is a covolume polynomial, where $W'$ and $W''$
are irreducible subvarieties of $\Pbb:=\Pbb^{n_0}\times \cdots\times\Pbb^{n_\ell}$.
For this, note that $W'\times W''$ is an irreducible subvariety of $\Pbb\times\Pbb$;
therefore, $P_{[W']} (\ut) P_{[W'']}(\uu)=P_{[W'\times W'']}(\ut,\uu)$ is a covolume
polynomial, and by diagonalization (that is, by Theorem~\ref{thm:nonnegcc}),
so is $P_{[W']} (\ut) P_{[W'']}(\ut)$.
\end{proof}

\begin{proof}[Proof of Theorem~\ref{thm:nonnegcc}]
By continuity, we may assume that $P(\ut)$ has rational coefficients and $A=(a_{ij})$ 
has positive rational entries; and up to a scalar multiple we can then assume that 
$A$ has positive integer entries and $P(\ut)=P_{[W]}(\ut)$ is the polynomial 
associated with an irreducible subvariety $W$ of $\Pbb:=(\Pbb^M)^{\ell+1}$, with 
$M\gg 0$. Explicitly,
\[
P(\ut)=\sum_{f_0+\cdots+f_\ell=d} \beta_{f_0\dots f_\ell} t_0^{f_0}\cdots t_\ell^{f_\ell}
\]
where
\[
[W]=\sum_{f_0+\cdots+f_\ell=d} \beta_{f_0\dots f_\ell} h_0^{f_0}\cdots h_\ell^{f_\ell}
\cap [(\Pbb^M)^{\ell+1}]
\]
and as usual $h_i$ denotes the pull-back of the hyperplane class from the corresponding
factor of $(\Pbb^M)^{\ell+1}$.

The codimension of $W$ is the degree $d$ of $P_{[W]}$. We choose an integer $S>d$,
implying in particular $(m+1)S>d$; and we may assume (cf.~Remark~\ref{rem:event}) that 
$M$ is sufficiently large to allow us to define Segre-Veronese embeddings
\[
\sigma_i:\quad (\Pbb^S)^{m+1} \to \Pbb^M
\]
for $i=0,\dots, \ell$, such that
\[
\sigma_i^*(h)=a_{i0} k_0+\cdots + a_{im} k_m
\]
where $h$ is the hyperplane class in $\Pbb^M$ and $k_0,\dots, k_m$ are the pull-backs
of the hyperplane classes from the factors. These embeddings determine an 
embedding 
\[
\varphi:\quad (\Pbb^S)^{m+1} \hookrightarrow (\Pbb^M)^{\ell+1}
\]
such that for $i=0,\dots,\ell$,
\[
\varphi^*(h_i)=a_{i0} k_0+\cdots + a_{im} k_m\saf.
\]
Now we may not have direct control of $\varphi^{-1}(W)$, but we can establish that 
the inverse image of a {\em translate\/} of $W$ is irreducible. 

\begin{claim}\label{cla:deb}
For a general $\gamma\in \PGL(M+1)^{\ell+1}$, the inverse image 
$\varphi^{-1}(^\gamma W)$ of the $\gamma$-translate of $W$ is irreducible.
\end{claim}

This claim follows from a result of Olivier Debarre, \cite[Th\'eor\`eme~2.2, 2)a))]{MR1420927}. 
To verify this, let
\[
[\varphi(\Pbb^S)^{m+1}] 
=\sum_{e_0,\dots, e_\ell} \alpha_{e_0\dots e_\ell} h_0^{e_0}\cdots h_\ell^{e_\ell}
\cap [(\Pbb^M)^{\ell+1}]
\]
where the sum is over all nonnegative $e_0,\dots, e_\ell$ such that 
$\sum_i (M-e_i)=(m+1)S$. The coefficients $\alpha_{e_0\dots e_\ell}$ equal the intersection
numbers
\begin{align*}
\alpha_{e_0\dots e_\ell} &= \int h_0^{M-e_0}\cdots h_\ell^{M-e_\ell}\cap [\varphi(\Pbb^S)^{m+1}] \\
&=\int \prod_{i=0}^\ell (a_{i1} k_1 + \cdots + a_{im} k_m)^{M-e_i}\cap
[(\Pbb^S)^{m+1}]
\end{align*}
and in particular $\alpha_{e_0\dots e_\ell}\ne 0$ for all nonnegative $e_0,\dots, e_\ell$
such that $\sum_i (M-e_i)=(m+1)S$ since the entries of $A$ are assumed to be positive.

By Debarre's theorem, $\varphi^{-1}(^\gamma W)$ is irreducible
for a general $\gamma$ if for all nonempty subsets $I\subseteq\{0,\dots, \ell\}$
there exist $\underline e$, $\underline f$ such that
\[
\alpha_{e_0\dots e_\ell}\ne 0,\quad
\beta_{f_0\dots f_\ell}\ne 0,\quad
\text{and}\quad 
\sum_{i\in I} (e_i+f_i)<|I|M\saf.
\]
As observed, $\alpha_{e_0\dots e_\ell}\ne 0$ for all $\underline e$ such that 
$\sum_i (M-e_i)=(m+1)S$. As $W$ has codimension~$d$, there exist $\underline f$ with
$\sum_j f_j=d$ such that $\beta_{f_0\dots f_\ell}\ne 0$; note that for all 
$I\subseteq\{0,\dots, \ell\}$, $\sum_{i\in I} f_i\le d$. 
Given a nonempty $I\subseteq \{0,\dots, \ell\}$, let $e_i=M$ for $i\not\in I$ and choose
any $e_i\ge 0$ for $i\in I$ s.t.~$\sum_{i\in I} e_i = |I|M-(m+1)S$. (We can do this because
$I$ is not empty and $M\ge (m+1)S$.) Then we have
\[
\sum_{i\in I} (e_i+f_i) \le |I|M-(m+1)S+d<|I|M
\]
as we are assuming $(m+1)S>d$.

Thus the hypothesis of Debarre's theorem is satisfied and Claim~\ref{cla:deb} follows.\qede

The class of $\varphi^{-1}(^\gamma W)$ is 
\[
[\varphi^{-1}(^\gamma W)]=
\varphi^*([W])
=P_{[W]}(a_{00} k_0+\cdots + a_{0m} k_m,\dots, a_{\ell 0} k_0+\cdots + a_{\ell m} k_m)\saf.
\]
The conclusion is that 
\[
P_{[W]}(A\uu) = P_{[\varphi^{-1}(^\gamma W)]}(\uu)
\]
is the polynomial associated with the irreducible $\varphi^{-1}(^\gamma W)$, hence 
a covolume polynomial, as needed.
\end{proof}

\begin{corol}\label{cor:slc}
Covolume polynomials are sectional log-concave.
\end{corol}

\begin{proof}
Given Definition~\ref{def:slc}, this is now an immediate consequence of 
Theorem~\ref{thm:nonnegcc} and Lemma~\ref{lem:twovac}.
\end{proof}

\begin{example}
(Cf.~Example~\ref{ex:notlor}.) The polynomial
\[
P(t_0,t_1,t_2):=
2t_0^7t_1^2 + 2t_0^6t_1^3 + 2t_0^6t_1^2t_2 + 2t_0^5t_1^4 + 4t_0^5t_1^3t_2 + 
6t_0^4t_1^4t_2
\]
as well as its normalization $N(P)$ are not Lorentzian, but $P$ {\em is\/} sectional log-concave,
as it is a covolume polynomial.
\qede\end{example}

\begin{remark}\label{rem:lcslc}
Let $P(u,v)$ be a homogeneous log-concave polynomial\footnote{Recall
that by this we mean that the coefficients of $P(u,v)$ form a log-concave sequence
{\em with no internal zeros.\/}}. By Lemma~\ref{lem:twovac},
$P$~is a covolume polynomial; by Theorem~\ref{thm:nonnegcc}, $P(A\binom uv)$ is
a covolume polynomial for all $2\times 2$ invertible matrices $A$ with nonnegative
entries; and then $P(A\binom uv)$ must be a log-concave polynomial, again by
Lemma~\ref{lem:twovac}. This particular case of Theorem~\ref{thm:nonnegcc}
implies that log-concave polynomials are indeed sectional log-concave, as claimed 
earlier in this section. 

One way to view this fact is as follows: for bivariate homogeneous polynomials, the 
property of having Lorentzian normalization is preserved by nonnegative coordinate
changes. We don't know if this is also the case for homogeneous polynomials in
more variables. It does not appear to be a direct consequence of the fact that
the property of being Lorentzian {\em is\/} preserved by nonnegative coordinate
changes (i.e., \cite[Theorem 2.10]{MR4172622}). 

The fact that the class of log-concave polynomials is preserved by nonnegative changes
of coordinates is easily seen to be equivalent to the assertion that if $f(t)$ is a (non-homogenous)
log-concave univariate polynomial, then so is $f(t+1)$. For alternative arguments proving
this fact, see~\cite[Theorem~2]{MR342424} and Corollary~8.4 in the survey \cite{MR1310575}.
\qede\end{remark}

\begin{remark}
As shown in Proposition~\ref{prop:covtolor}, covolume polynomials are `dually Lorentzian' 
in the sense of~\cite{dualLor}. 
Corollary~\ref{cor:slc} could be proved as a consequence of \cite[Theorem~1.3]{dualLor}. 
The proof given above is independent of~\cite{dualLor}.
\qede\end{remark}


\section{Segre classes}\label{partII}

We work over an algebraically closed field $k$; {\em schemes\/} are assumed to be of finite type
(but not necessarily reduced or irreducible).
The Segre class $s(Z,Y)$ of a closed embedding $Z\subseteq Y$ of schemes is a class in the 
Chow group $A_*(Z)$ recording important intersection-theoretic information about the embedding. 
For a thorough treatment of Segre classes, the reader is addressed to~\cite[Chapter~4]{85k:14004}.
By definition, the Segre class $s(Z,Y)$ is the Segre class of the cone $C_ZY$; in particular,
if $Z\subseteq Y$ is a regular embedding, then $s(Z,Y)=c(N_ZY)^{-1}\cap [Z]$ is the inverse
Chern class of the normal bundle of $Z$ in $Y$. Segre classes are preserved by birational
morphisms (\cite[Proposition~4.2(a)]{85k:14004}). These two properties characterize Segre classes:
the birational invariance reduces the computation of~$s(Z,Y)$ to the computation of $s(E,\Til Y)$,
where $E$ is the exceptional divisor in the blow-up $\Til Y$ of $Y$ along $Z$, and since $E$
is regularly embedded,
\[
s(E,\Til Y) = c(N_E\Til Y)^{-1}\cap [E] = \frac {[E]}{1+E}\saf,
\]
where
\[
\frac 1{1+E}=1-E+E^2-E^3+\cdots.
\]
(Cf.~\cite[Corollary~4.2.2]{85k:14004}.)

Segre classes are a key ingredient in the definition of the intersection product in Fulton-MacPherson 
intersection theory (\cite[Proposition~6.1(a)]{85k:14004}), with direct applications to enumerative
geometry. They can also be used to express important classical invariants such as multiplicity, 
Milnor numbers, local Euler obstructions, etc. For a survey of the role of Segre classes in the 
theory of singularities, see~\cite{MR4461013}.

Now let $I=\{F_i\}_{i=0,\dots, r}$ be a finite set of homogeneous polynomials in 
$k[x_0,\dots,x_n]$. For all $N\ge n$,
let $Z_N$ be the subscheme of $\Pbb^N$ defined by the ideal generated by $I$ 
in~$k[x_0,\dots, x_N]$; we may view $Z_N$ as a cone over $Z=Z_n$.
Let $\iota_N: Z_N\to \Pbb^N$ be the embedding. We consider the push-forwards 
$\iota_{N*} s(Z_N,\Pbb^N)$ of the corresponding Segre classes. 
In previous work we have observed that these classes are organized by a power series
\begin{equation}\label{eq:zeta}
\zeta_I(t):=\sum_{i\ge 0} s^{(i)} t^i
\end{equation}
with integer coefficients $s^{(i)}$, such that
\[
\iota_{N*} s(Z_N,\Pbb^N) = \sum_{i=0}^N s^{(i)} H^i\cap [\Pbb^N]\saf,
\]
where $H$ denotes the hyperplane class in $\Pbb^N$. The fact that $\zeta_I$ is well-defined
is the content of~\cite[Lemma~5.2]{MR3709134}.
The series $\zeta_I(t)$ only depends on the ideal generated by $I$ in~$k[x_0,\dots, x_n]$,
and should be viewed as an invariant of the ideal, but it is useful to keep track of the chosen
generating set $I$; this plays a role in the following statement.

\begin{theorem}[\cite{MR3709134}, Theorem~5.8]\label{thm:szf}
The power series $\zeta_I(t)$ is rational. More precisely, let $I=\{F_i\}_{i=0,\dots, r}$ 
and let $d_i=\deg F_i$; then
\[
\zeta_I(t)=\frac{P(t)}{\prod_{i=0}^r (1+d_i t)}
\]
where $P(t)\in \Zbb[t]$ is a polynomial with nonnegative coefficients, trailing term of degree
$\codim(I)$, and leading term $\prod_{i=0}^r d_i t^{r+1}$.
\end{theorem}

As a consequence of this result, the polynomial
\[
\prod_{i=0}^r (1+d_i t) - P(t)
\]
has degree $\le r$. In this section we will prove that {\em the coefficients of this polynomial
form a log-concave sequence of nonnegative integers with no internal zeros.\/} 
We will further prove that the coefficients of $P(t)$ also form a log-concave sequence of
nonnegative integers with no internal zeros, provided that the normal cone of $Z$ in $\Pbb^n$
is irreducible.

We will view the set-up described above as a particular case of the following more general
situation. Let $Z\subseteq \Pbb^{n_1}\times\cdots\times \Pbb^{n_\ell}$ be a closed subscheme,
whose ideal $I$ is generated by a finite set of multihomogeneous polynomials.
We consider the subscheme $Z_{\uN}$ defined by the same polynomials in 
$\Pbb^{\uN}:=\Pbb^{N_1}\times\cdots\times \Pbb^{N_\ell}$, for all $\uN=(N_1,\dots, N_\ell)$ 
with $N_j\ge n_j$. The analogue of~\eqref{eq:zeta} is a power series
\begin{equation}\label{eq:szfell}
\zeta_I(t_1,\dots, t_\ell) :=\sum_{i\ge 0} s^{(i_1\dots i_\ell)} t_1^{i_1}\cdots t_\ell^{i_\ell}
\end{equation}
such that for all $\uN$ as above, and denoting by $\iota_{\uN}$ the inclusion $Z_{\uN}
\hookrightarrow \Pbb^{\uN}$ and by $h_j$ the pull-back of the hyperplane class from 
the $j$-th factor, 
\begin{equation}\label{eq:szfprop}
\iota_{\uN *} s(Z_{\uN},\Pbb^{\uN}) = \sum_{i_j\le N_j} s^{(i_1\dots i_\ell)} 
h_1^{i_1}\cdots h_\ell^{i_\ell}\cap [\Pbb^{\uN}]\saf.
\end{equation}
The natural extension of Theorem~\ref{thm:szf} holds for this power series.

\begin{theorem}\label{thm:szf2}
The power series $\zeta_I(t_1,\dots,t_\ell)$ is rational. More precisely, let 
$(e_{1k},\dots, e_{\ell k})$, $k=0,\dots,r$, be the multidegrees of the elements of~$I$, 
and let
\[
Q(t_1,\dots, t_\ell)=\prod_{k=0}^r (1+e_{1k} t_1+\cdots+e_{\ell k} t_\ell)\saf.
\]
Then
\[
\zeta_I(t_1,\dots, t_\ell)=\frac{P(t_1,\dots, t_\ell)}{Q(t_1,\dots, t_\ell)}
\]
where $P(t_1,\dots, t_\ell)\in \Zbb[t_1,\dots, t_\ell]$ is a polynomial with nonnegative 
coefficients, trailing term of total degree $\codim(I)$, and leading term equal to
 $\prod_{k=0}^r (e_{1k} t_1+\cdots+e_{\ell k} t_\ell)$.
\end{theorem}

A proof of this result may be obtained by following the same blueprint as the proof of
Theorem~\ref{thm:szf} given in~\cite{MR3709134}; for the case $\ell=2$, also
see~\cite[\S5.2]{MR4099922}. 

We focus on the polynomial
\[
R(t_1,\dots,t_\ell)=Q(t_1,\dots, t_\ell)-P(t_1,\dots, t_\ell)\saf,
\]
for which
\[
1-\zeta_I(t_1,\dots, t_\ell)=\frac{R(t_1,\dots, t_\ell)}{Q(t_1,\dots, t_\ell)}\saf.
\]
As a consequence of Theorem~\ref{thm:szf2}, $R(t_1,\dots, t_\ell)$ is a polynomial with
integer coefficients and total degree $\le r$. Our main goal is to establish log-concavity
properties of these polynomials. As they are not necessarily homogeneous, we adapt 
Definition~\ref{def:slc} accordingly.

\begin{defin}\label{def:slc2}
A polynomial $f(t_1,\dots,t_\ell)\in\Rbb[t_1,\dots,t_\ell]$ is {\em sectional log-concave\/}
if for all $\ell\times 2$ matrices $A$ with nonnegative real entries, the polynomial 
$p(A\binom 1v)$ is either identically~$0$ or its coefficients form a log-concave sequence
of nonnegative real numbers with no internal zeros.
\qede\end{defin}

\begin{remark}\label{rem:homdehom}
If $f$ is homogeneous, then it is sectional log-concave in the sense of Definition~\ref{def:slc2}
if and only if it is in the sense of Definition~\ref{def:slc}. For any $f(t_1,\dots,t_\ell)$, 
$f$ is sectional log-concave if any homogenization $F(t_0,\dots, t_\ell)$ of $f$
is sectional log-concave. Indeed, assume $F$ is homogeneous and $f=F|_{t_0=1}$.
Then $f(A\binom 1v)=F(u,A\binom uv)|_{u=1}$, and $F(u,A\binom uv)$ is log-concave
or identically $0$ if $F$ is sectional log-concave.
\qede\end{remark}

If $f$ is obtained by setting one of the variables of a homogeneous polynomial $F$ to $1$, 
we will say that $f$ is a {\em de-homogenization\/} of $F$.
We will say that $f(t_1,\dots,t_\ell)\in\Rbb[t_1,\dots,t_\ell]$ is {\em M-convex\/}
if it is the {de-homoge\-nization} of an M-convex polynomial.

A precise result can be established for the polynomials $R(t_1,\dots, t_\ell)$ arising
as numerators of $1-\zeta_I$.

\begin{theorem}\label{thm:R}
Let $I$ be a finite set of multihomogeneous forms of multidegrees 
$(e_{1k},\dots, e_{\ell k})$, $k=0,\dots,r$.
Define the polynomial $R(t_1,\dots,t_\ell)$ by the identity
\[
1-\zeta_I(t_1,\dots, t_\ell)=\frac{R(t_1,\dots, t_\ell)}
{\prod_{k=0}^r (1+e_{1k} t_1+\cdots+e_{\ell k} t_\ell)}\saf.
\]
Then $R(t_1,\dots, t_\ell)$ is the de-homogenization of a covolume polynomial
of degree~$r$. In particular:
\begin{itemize}
\item $R(t_1,\dots,t_\ell)$ has nonnegative coefficients and $M$-convex support;
\item $R(t_1,\dots,t_\ell)$ is sectional log-concave;
\item If $R=\sum_{0\le i_j \le r_j} b_{i_1\dots i_\ell} t_1^{i_1}\cdots t_\ell^{i_\ell}$, then
the normalization of the polynomial
\[
\sum_{0\le i_j \le r_j} b_{i_1\dots i_\ell} u_0^{i_1+\cdots +i_\ell} u_1^{r_1-i_1}\cdots 
u_\ell^{r_\ell-i_\ell}
\]
is Lorentzian;
\item In the univariate case, i.e., $\ell=1$, the coefficients of $R(t_1)$ form a log-concave 
sequence with no internal zeros.
\end{itemize}
\end{theorem}

It is natural to ask whether the numerator of $\zeta_I$ itself satisfies the same constraints. 
We can prove that this is the case, but only subject to an irreducibility hypothesis.

\begin{theorem}\label{thm:P}
Let $Z\subseteq \Pbb^{n_1}\times\cdots\times \Pbb^{n_\ell}$ be a closed subscheme whose 
ideal is generated by a set $I$ of multihomogeneous forms of multidegrees 
$(e_{1k},\dots, e_{\ell k})$, $k=0,\dots,r$.
Define the polynomial $P(t_1,\dots,t_\ell)$ by the identity
\[
\zeta_I(t_1,\dots, t_\ell)=\frac{P(t_1,\dots, t_\ell)}
{\prod_{k=0}^r (1+e_{1k} t_1+\cdots+e_{\ell k} t_\ell)}\saf.
\]
Assume that the projectivized normal cone 
$\Pbb(C_Z (\Pbb^{n_1}\times\cdots\times \Pbb^{n_\ell}))$ is
irreducible. Then $P(t_1,\dots, t_\ell)$ is the de-homogenization of a covolume polynomial
of degree~$r+1$.
In particular, 
\begin{itemize}
\item $P(t_1,\dots,t_\ell)$ has nonnegative coefficients and $M$-convex support; 
\item $P(t_1,\dots,t_\ell)$ is sectional log-concave;
\item If $P=\sum_{0\le i_j \le r_j} a_{i_1\dots i_\ell} t_1^{i_1}\cdots t_\ell^{i_\ell}$, then
the normalization of the polynomial
\[
\sum_{0\le i_j \le r_j} a_{i_1\dots i_\ell} u_0^{i_1+\cdots +i_\ell} u_1^{r_1-i_1}\cdots 
u_\ell^{r_\ell-i_\ell}
\]
is Lorentzian;
\item In the univariate case, i.e., $\ell=1$, the coefficients of $P(t_1)$ form a log-concave 
sequence with no internal zeros.
\end{itemize}
\end{theorem}

We will informally refer to the polynomials $P$, resp.~$R$, in these statements as the `numerators'
of $\zeta_I$, resp.~$1-\zeta_I$. We note that the series $\zeta_I$ only depends on the ideal
generated by $I$, while these polynomials depend on the chosen set $I$ of generators for
this ideal. 

\begin{example}\label{ex:charnos}
Segre classes have applications in enumerative geometry; the prototypical example is
the computation of the number 3264 of smooth plane conics that are tangent to five general
smooth conics, cf.~\cite[Examples~9.1.8, 9.1.9]{85k:14004}. 

The {\em characteristic numbers\/} for the family of smooth plane curves of degree~$d$ are 
the numbers of such curves that contain a selection of general points and are tangent to
a selection of general lines. For plane cubics, the characteristic numbers are
\[
1,\quad 4,\quad 16,\quad 64,\quad 976,\quad 3424,\quad 9766,\quad 21004,\quad 33616
\]
(\cite{Maillard,MR951637,MR1108634}): for example, there are $\oldstylenums{33{,}616}$ 
cubics tangent to~$9$ lines in general position.
The information of the characteristic numbers is equivalent to the information of the push-forward 
to the $\Pbb^9$ of
plane cubics of the Segre class of a scheme naturally supported on the set of non-reduced 
curves:
\[
48[\Pbb^4] - 480[\Pbb^3] + 3930[\Pbb^2]  - 38220[\Pbb^1]  + 372960[\Pbb^0]\saf.
\]
This scheme is cut out by a set $I$ of ten {\em quartic\/} hypersurfaces, and it
follows that the corresponding Segre zeta function is
\begin{equation}\label{eq:Szcub}
\zeta_I(t):=\frac{48t^5+1440t^6+19290t^7+142020t^8+567840t^9+1048576t^{10}}{(1+4t)^{10}}\saf.
\end{equation}
The numerator of $1-\zeta_I(t)$ equals
\[
1+40t + 720t^2 + 7680t^3 + 53760t^4 + 258000t^5 + 858720t^6 + 1946790t^7 + 2807100t^8 + 2053600t^9
\]
and as prescribed by Theorem~\ref{thm:R} the coefficients of this polynomial form a log-concave
sequence with no internal zeros.
\qede\end{example}

\begin{remark}
The characteristic numbers may be interpreted as the multidegrees of the closure of the 
graph of the duality map (in the case of cubics, the map associating to a smooth cubic the
corresponding dual sextic). It follows that they form a log-concave sequence of integers 
with no internal zeros, by~\cite[Theorem~21]{MR2904577}.
\qede\end{remark}

Our main tool in the proof of Theorems~\ref{thm:R} and~\ref{thm:P} will be a
result providing a large supply of polynomials that are de-homogenizations of covolume 
polynomials.

\begin{prop}\label{prop:irremap}
Let $X$ be an irreducible variety and let
\[
q: X \to \Pbb^{n_1}\times\cdots\times \Pbb^{n_\ell}
\]
be a proper map. Let $\cR$ be a globally generated vector bundle of rank $r$ on $X$, and write
\begin{equation}\label{eq:ccEZL}
q_*(c(\cR)\cap [X]) = \sum_{0\le i_j\le n_j} a_{i_1\dots i_\ell} h_1^{i_1}\cdots h_\ell^{i_\ell}
\cap [\Pbb^{n_1}\times\cdots \times \Pbb^{n_\ell}]\saf.
\end{equation}
Then the polynomial
\begin{equation}\label{eq:dehomcv}
\sum_{0\le i_j\le n_j} a_{i_1\dots i_\ell} t_1^{i_1}\cdots t_\ell^{i_\ell}
\end{equation}
is the de-homogenization of a covolume polynomial of degree $\rkR+(\sum n_j) -\dim X$.
\end{prop}

\begin{proof}
Let $V$ be a $(D+1)$-dimensional vector space of global sections 
generating~$\cR$. Consider the product
\[
\xymatrix{
& \Pbb(V)\times X \ar[ld]_{p_1} \ar[rd]^{p_2} \\
\Pbb(V) & & X
}\saf,
\]
and denote by $H$ the hyperplane class in $\Pbb(V)$. 
Let $\cK$ be the kernel of the 
surjection $\cV=V\otimes \cO_X \to \cR$.
Identify $\Pbb(V)\times X$ with~$\Pbb(\cV)$ and view $\Pbb(\cK)$ as a codimension-$\rkR=\rk\cR$
subscheme of $\Pbb(V)\times X$; also note that $p_1^*H=c_1(\cO_\cV(1))$.
By definition, $\Pbb(\cK)$ is the zero-scheme of the composition
\[
\cO_\cV(-1) \hookrightarrow p_2^*\cV \twoheadrightarrow p_2^*\cR\saf;
\]
therefore, of the corresponding section $\cO\to p_2^*\cR \otimes \cO_\cV(1)$.
This section is regular (\cite[B.5.6]{85k:14004}), and it follows 
that the class $[\Pbb(\cK)]$ in $A_*(\Pbb(V)\times X)$ is the top Chern class of 
$p_2^*\cR\otimes \cO_\cV(1)$:
\begin{equation}\label{eq:larcla}
[\Pbb(\cK)]=\sum_{i=0}^\rkR (p_1^* H^{\rkR-i}) (p_2^* c_i(\cR))\cap [\Pbb(V)\times X]
\end{equation}
in $A_{D+\dim X-r}(\Pbb(V)\times X)$. Next, consider the proper map
\[
\xymatrix{
\Pbb(V)\times X \ar[rr]^-{\id\times q} && \Pbb(V)\times \Pbb^{n_1}\times\cdots\times \Pbb^{n_\ell}\saf.
}
\]
We claim that
\begin{equation}\label{eq:clmWup}
(\id\times q)_*([\Pbb(\cK)]) = \sum_{i=0}^\rkR \sum_{\sum i_j=i-\dim X+\sum n_j} a_{i_1\dots i_\ell} 
H^{\rkR-i} h_1^{i_1}\cdots h_\ell^{i_\ell}\cap
[\Pbb(V)\times \Pbb^{n_1}\times\cdots\times \Pbb^{n_\ell}]
\end{equation}
with $a_{i_1 \dots i_\ell}$ as in~\eqref{eq:ccEZL}, and where $H, h_j$ denote the pull-backs
of the corresponding classes to the product 
$\Pbb(V)\times \Pbb^{n_1}\times\cdots\times \Pbb^{n_\ell}$.

To prove~\eqref{eq:clmWup}, consider the diagram
\[
\xymatrix@C=10pt@R=10pt{
& \Pbb(V)\times X \ar[ld]_{p_1} \ar[rr]^-{p_2} \ar[dd]^{\id\times q} & & X \ar[dd]^q \\
\Pbb(V) \\
& \Pbb(V)\times \Pbb^{n_1}\times\cdots\times \Pbb^{n_\ell} \ar[lu]_{\pi_1} \ar[rr]^-{\pi_2} 
& & \Pbb^{n_1}\times\cdots\times \Pbb^{n_\ell}
}
\]
We have
\[
p_1^* = (\id \times q)^*\circ \pi_1^*
\quad\text{and}\quad
(\id \times q)_* \circ p_2^* = \pi_2^* \circ q_*
\]
(\cite[Proposition~1.7]{85k:14004}). By~\eqref{eq:larcla}, these identities and the projection 
formula give
\begin{align*}
(\id\times q)_*([\Pbb(\cK)]) &=\sum_{i=0}^\rkR (\id\times q)_*((p_1^* H^{\rkR-i}) (p_2^* c_i(\cR))\cap
[\Pbb(V)\times X]) \\
&=\sum_{i=0}^\rkR (\pi_1^* H^{\rkR-i}) \cap \pi_2^* q_*(c_i(\cR)\cap [X])\saf.
\end{align*}
Using~\eqref{eq:ccEZL}, we see that this class equals
\[
\sum_{i=0}^\rkR (\pi_1^* H^{\rkR-i})\cap \pi_2^* \sum_{\sum i_j = i-\dim X+\sum n_j}
a_{i_1\dots i_\ell} h_1^{i_1}\cdots h_\ell^{i_\ell} \cap
[\Pbb^{n_1}\times\cdots\times \Pbb^{n_\ell}]
\]
and~\eqref{eq:clmWup} follows. 

Since by definition $(\id\times q)_*([\Pbb(\cK)])$ is a multiple of the class of the irreducible 
subvariety $(\id\times q)(\Pbb(\cK))$, this shows that
\[
\sum_{i=0}^\rkR \sum_{\sum i_j=i-\dim X+\sum n_j} a_{i_1\dots i_\ell} 
t_0^{\rkR-i} t_1^{i_1}\cdots t_\ell^{i_\ell}
\]
is a covolume polynomial. The polynomial~\eqref{eq:dehomcv} is obtained from this 
polynomial by setting $t_0=1$, so this proves the statement.
\end{proof}

\begin{example}
The degrees of the Chern classes of a globally generated bundle over projective space
form a log-concave sequence of nonnegative integers with no internal zeros.
This follows from Proposition~\ref{prop:irremap} and Lemma~\ref{lem:twovac} in the very 
particular case where $\ell=1$ and $q$ the identity map $\Pbb^{n_1} \to \Pbb^{n_1}$. 
\qede\end{example}

Theorems~\ref{thm:R} and~\ref{thm:P} are consequences of Proposition~\ref{prop:irremap}.

\begin{proof}[Proof of Theorems~\ref{thm:R} and~\ref{thm:P}]
In both statements, the four listed properties are formal consequence of the assertion
that the polynomial is the de-homogenization of a covolume polynomial. Specifically, 
the first property follows from Corollary~\ref{cor:Mconv}; the second from 
Corollary~\ref{cor:slc} (cf.~Remark~\ref{rem:homdehom}); the third from 
Proposition~\ref{prop:covtolor}; and the last property from Lemma~\ref{lem:twovac}.

Therefore, it suffices to prove that $R(t_1,\dots, t_\ell)$ is a de-homogenization of a
covolume polynomial, and so is $P(t_1,\dots, t_\ell)$ if 
$\Pbb(C_Z(\Pbb^{n_1}\times\cdots\times \Pbb^{n_\ell}))$ is irreducible.

Let $\uN=(N_1,\dots, N_\ell)$, with $N_j\ge n_j$ for all $j$.
Denote by $\cO_j(1)$ the pull-back of the hyperplane line bundle from the $j$-th factor
of $\Pbb^{\uN}$. The subscheme $Z_{\uN}$ of $\Pbb^{\uN}$ is cut out by hypersurfaces
$X_k$, $k=0,\dots,r$, with $\cO(X_k)=\cO_1(e_{1k})\otimes\cdots \otimes \cO_\ell(e_{\ell k})$.
We have the fiber square
\begin{equation}\label{eq:CD1}
\begin{gathered}
\xymatrix{
Z_{\uN} \ar@{^(->}[r]^-{\iota_{\uN}} \ar[d]_\delta & \Pbb^{\uN} \ar[d]^\Delta \\
X_0\times\cdots\times X_r \ar@{^(->}[r] & \Pbb^{\uN}\times\cdots\times \Pbb^{\uN}
}
\end{gathered}
\end{equation}
where $\Delta$ is the diagonal embedding.
This is an instance of the situation considered in~\cite[\S6.1]{85k:14004};
the Fulton-MacPherson intersection product $(X_0\times\cdots\times X_r)\cdot \Pbb^{\uN}$
is one term in the class
\[
\delta^*(N_{X_0\times\cdots\times X_r}(\Pbb^{\uN}\times\cdots\times \Pbb^{\uN}))
\cap s(Z_{\uN},\Pbb^{\uN})
=c(\iota_{\uN}^* \cN)\cap s(Z_{\uN},\Pbb^{\uN})\saf,
\]
where
\[
\cN=\oplus_{k=0}^r \left(\cO_1(e_{1k})\otimes\cdots \otimes \cO_\ell(e_{\ell k})\right)\saf.
\]
By~\cite[Example~6.1.6]{85k:14004}, this class only has terms in codimension~$\le r+1$.
Thus, its push-forward
\begin{align*}
\iota_{\uN*} \big( c(\iota_{\uN}^* \cN)\cap s(Z_{\uN},\Pbb^{\uN}) \big)
&=c(\cN)\cap \iota_{\uN*}s(Z_{\uN},\Pbb^{\uN}) \\
&=\prod_{k=0}^r (1+e_{1k} h_1+\cdots+e_{\ell k} h_\ell) \cap \iota_{\uN *} s(Z_{\uN},\Pbb^{\uN})
\end{align*}
may be written as a polynomial 
\[
\sum_{0\le i_j\le N_j} a_{i_1\dots i_\ell} h_1^{i_1}\cdots h_\ell^{i_\ell} \cap [\Pbb^{\uN}]
\]
of total degree $\le (r+1)$. 
By~\eqref{eq:szfell}, \eqref{eq:szfprop}, and Theorem~\ref{thm:szf2},
\[
P(h_1,\dots, h_\ell)=\prod_{k=0}^r (1+e_{1k} h_1+\cdots+e_{\ell k} h_\ell) \cap \iota_{\uN *} 
s(Z_{\uN},\Pbb^{\uN})
\]
in $A_*(\Pbb^{\uN})$, that is, modulo $h_j^{N_j+1}$ for all $j$. Taking $N_j\ge\max(n_j,r+1)$,
we get the equality {\em of polynomials\/}
\[
P(t_1,\dots, t_\ell)=\sum_{0\le i_j\le N_j} a_{i_1\dots i_\ell} t_1^{i_1}\cdots t_\ell^{i_\ell}
\in \Zbb[t_1,\dots, t_\ell]\saf.
\]
Summarizing, let $N_j\gg 0$ for all $j$; then the polynomial 
$P(t_1,\dots, t_\ell)\in \Zbb[t_1,\dots, t_\ell]$ is the unique lift of degree $\le N_j$ in $t_j$
of the class
\begin{equation}\label{eq:clasP}
c(\cN)\cap \iota_{\uN*} s(Z_{\uN},\Pbb^{\uN})\saf.
\end{equation}
It also follows that $R(t_1,\dots, t_\ell)$ is the unique lift of degree $\le N_j$ in $t_j$
of the class
\begin{equation}\label{eq:clasR}
c(\cN)\cap (1-\iota_{\uN*} s(Z_{\uN},\Pbb^{\uN}))\saf.
\end{equation}

Next, consider the blow-up $\Til \Pbb^{\uN}$ of $\Pbb^{\uN}$ along $Z_{\uN}$, with 
exceptional divisor $E=\Pbb(C_{Z_{\uN}}\Pbb^{\uN})$ and notation as in the following diagram:
\begin{equation}\label{eq:CD2}
\begin{gathered}
\xymatrix{
E \ar[r]^-j \ar[d]_\rho & \Til \Pbb^{\uN} \ar[d]^\nu \\
Z_{\uN} \ar[r]_{\iota_{\uN}} & \Pbb^{\uN}
}
\end{gathered}
\end{equation}
By the birational invariance of Segre classes, 
\[
\iota_{\uN*} s(Z_{\uN},\Pbb^{\uN}) = \nu_* j_* s(E,\Til \Pbb^{\uN}) \\
= \nu_* j_* \left( c(N_E\Til \Pbb^{\uN})^{-1}\cap [E] \right)\saf.
\]
Stacking diagrams~\eqref{eq:CD1} and~\eqref{eq:CD2} gives a fiber square
\[
\xymatrix{
E \ar[r]^-j \ar[d]_{\delta\circ \rho} & \Til \Pbb^{\uN} \ar[d]^{\Delta\circ\nu} \\
X_0\times\cdots\times X_r \ar@{^(->}[r] & \Pbb^{\uN}\times\cdots\times \Pbb^{\uN}
}
\]
whose excess intersection bundle (in the sense of~\cite[\S6.3]{85k:14004}) is
\[
(\delta\circ\rho)^*N_{X_0\times\cdots\times X_r} 
(\Pbb^{\uN}\times\cdots\times \Pbb^{\uN})/N_E\Til \Pbb^{\uN}
=\rho^*\iota_{\uN}^* \cN/j^*\cO(E) = j^*\cR
\]
with
\[
\cR=\nu^*(\oplus_{k=0}^r \left(\cO_1(e_{1k})\otimes\cdots \otimes \cO_\ell(e_{\ell k})\right))/\cO(E)
\saf.
\]
With this notation, and repeatedly using the projection formula, 
the class in~\eqref{eq:clasP} (represented by the polynomial $P(t_1,\dots, t_\ell)$) 
may be rewritten as
\begin{align*}
c(\cN)\cap \iota_{\uN*} s(Z_{\uN},\Pbb^{\uN}) &= 
c(\cN)\cap \nu_* j_* \left( c(N_E\Til \Pbb^{\uN})^{-1}\cap [E] \right) \\
&= (\nu\circ j)_*\left( c(j^*\nu^*\cN) c(N_E\Til \Pbb^{\uN})^{-1}\cap [E] \right) \\
&= (\nu\circ j)_*\left( c(\rho^* \iota_{\uN}^*\cN) c(j^*\cO(E))^{-1}\cap [E] \right) \\
&= (\nu\circ j)_*\left( c(j^*\cR)\cap [E] \right) \saf.
\end{align*}
Since $\cR$ is generated by global sections, Proposition~\ref{prop:irremap} implies that
{\em if $E=\Pbb(C_{Z_{\uN}}\Pbb^{\uN})$ is irreducible, then\/} $P(t_1,\dots, t_\ell)$ is
the de-homogenization of a covolume polynomial. This establishes the first part of
Theorem~\ref{thm:P}, since $\Pbb(C_{Z_{\uN}}\Pbb^{\uN})$ is irreducible if and only if
$\Pbb(C_Z (\Pbb^{n_1}\times\cdots\times \Pbb^{n_\ell}))$ is (recall that $Z_{\uN}$ is a
cone over $Z$).

Concerning the class in~\eqref{eq:clasR}, represented by the polynomial 
$R(t_1,\dots, t_\ell)$:
\begin{align*}
c(\cN)\cap (1-\iota_{\uN*} s(Z_{\uN},\Pbb^{\uN})) &= 
c(\cN)\cap \nu_*\left(1- \frac{ [E]}{1+E} \right) \\
&= c(\cN)\cap \nu_*\frac{[\Til\Pbb^{\uN}]}{1+E} \\
&= \nu_*\left( c(\nu^*\cN) c(\cO(E))^{-1}\cap [\Til\Pbb^{\uN}]\right) \\
&= \nu_*\left( c(\cR)\cap [\Til\Pbb^{\uN}]\right)\saf.
\end{align*}
Since $\cR$ is generated by global sections and $[\Til \Pbb^{\uN}]$ is
irreducible, Proposition~\ref{prop:irremap} implies (unconditionally) that
$R(t_1,\dots, t_\ell)$ is the de-homogenization of a covolume polynomial,
establishing the first part of Theorem~\ref{thm:R} and completing the proof.
\end{proof}

As proved above, the polynomials $P$ and $R$ determined by (a set of
polynomials defining) a closed subscheme
$Z\subseteq \Pbb^{n_1}\times\cdots\times \Pbb^{n_\ell}$ are closely related to 
Lorentzian polynomials. It is natural to inquire whether their homogenizations may 
themselves be Lorentzian, perhaps after normalization. 

For $\ell=1$, the {\em normalization\/} of the numerator $R$ of $1-\zeta_I$ is Lorentzian, 
as a consequence of Theorem~\ref{thm:R} and of \cite[Example~2.26]{MR4172622}.
It is not necessarily Lorentzian itself.

\begin{example}\label{ex:RnotLor}
Let $Z=\Pbb^{n-3}$ viewed as the subscheme of $\Pbb^n$ defined by the ideal of
$k[x_0,\dots,x_n]$ generated by $I=\{x_0,x_1,x_2\}$. 
If $H$ denotes the hyperplane class in $\Pbb^n$, the Segre class $s(Z,\Pbb^n)=
c(N_{\Pbb^{n-3}}\Pbb^n)^{-1}\cap [\Pbb^{n-3}]$ pushes forward to $H^3/(1+H)^3\cap [\Pbb^n]$,
and it follows that
\[
1-\zeta_I(t_1)=1-\frac{t_1^3}{(1+t_1)^3}=\frac{1+3t_1+3t_1^2}{(1+t_1)^3}\saf.
\]
The homogenization $t_0^2+3t_0t_1+3t_1^2$ is not Lorentzian, i.e., the sequence $1,3,3$
is not {\em ultra-}log-concave.
Its normalization {\em is\/} Lorentzian, i.e., the sequence is log-concave, as prescribed by 
Theorem~\ref{thm:R}.

For another example, 
let $I=\{x_0,y_0,z_0\}\subseteq k[x_0,\dots, x_n; y_0,\dots, y_n; z_0,\dots, z_n)$, defining 
an inclusion of $\Pbb^{n-1}\times\Pbb^{n-1}\times\Pbb^{n-1}$ in $\Pbb^n\times\Pbb^n\times\Pbb^n$, 
for all $n\ge 1$. This is also a complete intersection, and it follows that
\[
1-\zeta_I(t_1,t_2,t_3)=1-\frac{t_1 t_2 t_3}{(1+t_1)(1+t_2)(1+t_3)}
=\frac{1+t_1+t_2+t_3+t_1 t_2+t_1 t_3+t_2 t_3}{(1+t_1)(1+t_2)(1+t_3)}\saf,
\]
so that the homogenization of $R(t_1,t_2,t_3)$ is
\begin{equation}\label{eq:Rex}
t_0^2+t_0t_1+t_0t_2+t_0t_3+t_1 t_2+t_1 t_3+t_2 t_3\saf.
\end{equation}
This polynomial is sectional log-concave and M-convex (as prescribed by Theorem~\ref{thm:R}), 
but it is not Lorentzian. We note that its normalization
\[
\frac 12 t_0^2+t_0t_1+t_0t_2+t_0t_3+t_1 t_2+t_1 t_3+t_2 t_3
\]
{\em is\/} Lorentzian.
\qede\end{example}

We know of no example of a closed subscheme 
$Z\subseteq \Pbb^{n_1}\times\cdots\times \Pbb^{n_\ell}$
for which the normalization of the homogenization of the numerator of $1-\zeta_I$
is not Lorentzian. This does not appear to follow directly from Theorem~\ref{thm:R}.
We pose it as a question.

\begin{quest}\label{que:norLor}
Is the {\em normalization\/} of the homogenization of the numerator of $1-\zeta_I$ always 
a Lorentzian polynomial?
\end{quest}

The conjecture stated in~\S\ref{S:intro} proposes that the answer to this question should 
be affirmative.
We have verified that this is the case for several hundred randomly chosen monomial ideals.

Concerning the polynomial $P$ of Theorem~\ref{thm:P}, that is, the `numerator of $\zeta_I$',
we note that the hypothesis of irreducibility of the normal cone cannot be removed.

\begin{example}\label{ex:Pnotcov}
Let $I=\{x_0 y_0,x_0 z_0\}\subseteq k[x_0,\dots, x_n; y_0,\dots, y_n; z_0,\dots, z_n]$, 
defining a closed subscheme $Z\subseteq \Pbb:=\Pbb^n\times\Pbb^n\times\Pbb^n$, $n\gg 0$.
The support of this subscheme is the union of an irreducible divisor and a codimension-two 
subvariety; $Z$ is not irreducible, and therefore its normal cone is not irreducible.
The function $\zeta_I(t_1,t_2,t_3)$ may be determined as follows: with evident notation,
the Segre class of $Z$ is $(h_1+\text{h.o.t.})\cap [\Pbb]$, since this is the case away from
the codimension-$2$ component; since $I$ is generated by divisors with classes $h_1+h_2$,
$h_1+h_3$, we must have
\[
\zeta_I(t_1,t_2,t_3)=\frac{t_1+(t_1+t_2)(t_1+t_3)}{(1+t_1+t_2)(1+t_1+t_3)}
\]
by Theorem~\ref{thm:szf2}.
(For an alternative argument, compute the Segre class of $Z$ in $(\Pbb^n)^3$ by `residual
intersection', \cite[Proposition~9.2]{85k:14004} or~\cite[Proposition~3]{MR96d:14004}; and
then let $n\to \infty$.) The homogenization of the numerator $P(t_1,t_2,t_3)$ is
\begin{equation}\label{eq:Pcoun}
t_0 t_1 + t_1^2 +t_1 t_2+t_1t_3+t_2t_3\saf;
\end{equation}
the support of this polynomial consists of the points 
\[
(1,1,0,0)\quad,\quad
(0,2,0,0)\quad,\quad
(0,1,1,0)\quad,\quad
(0,1,0,1)\quad,\quad
(0,0,1,1)
\]
and is {\em not\/} M-convex: for $\alpha=(1,1,0,0)$ and $\beta=(0,0,1,1)$ we have $\alpha_0
>\beta_0$ and there is no $j$ such that $\beta_j>\alpha_j$ and $\alpha-e_0+e_j$, 
$\beta-e_j+e_0$ are both in $S$. This polynomial is also not sectional log-concave:
setting $t_0=4u,t_1=u,t_2=v,t_3=v$ gives $5u^2+2uv+v^2$, and $2^2\not\ge 5\cdot 1$.

By Proposition~\ref{prop:covtolor}, the numerator is not the de-homogenization of a
covolume polynomial. 
\qede\end{example}

Since the homogenization of the polynomial $P$ from Example~\ref{ex:Pnotcov} is not M-convex,
it also follows that it is not Lorentzian and neither is its normalization. In fact, the following
simple example shows that the numerator of $\zeta_I$ is not necessarily Lorentzian before 
normalization even in the univariate (i.e., $\ell=1$) case. 
(The Segre zeta function~\eqref{eq:Szcub} in Example~\ref{ex:charnos} also provides an example.)

\begin{example}
For any $n$, let $Z\subseteq \Pbb^n$ consist of a hyperplane $\Pbb^{n-1}$ with an embedded
component along the transversal intersection of two smooth quadric hypersurfaces in this 
hyperplane. More precisely, let $Z$ be the subscheme of $\Pbb^n$ defined by 
$I=\{x_0^2, x_0 Q_1,x_0 Q_2\} \subseteq k[x_0,\dots, x_n]$, where $Q_1$ and 
$Q_2$ are general homogeneous quadratic polynomials. The Segre class of $Z$ may be
computed by residual intersection, and this yields
\[
\zeta_I(t_1)=\frac{t_1+7t_1^2+18t_1^3}{(1+2 t_1)(1+3t_2)^2}\saf.
\]
The polynomial $t_1+7t_1^2+18t_1^3$ is not ultra-log-concave, so the homogenization
\[
t_0^2 t_1 + 7 t_0 t_1^2 + 18 t_1^3
\]
is not Lorentzian. The same polynomial is log-concave, therefore the {\em normalization\/}
of the homogenization is Lorentzian.
\qede\end{example}

It would be interesting to establish to what extent the irreducibility hypothesis in
Theorem~\ref{thm:P} can be weakened. 

\begin{quest}\label{quest:numsz}
For what subschemes of a product of projective spaces
is the numerator of $\zeta_I$ necessarily sectional log-concave?
\end{quest}

B.~Story~\cite{story} has verified that the numerator of $\zeta_I$ {\em is\/} log-concave 
for several families of subschemes $Z\subseteq \Pbb^n$ (that is, the $\ell=1$ case) 
not satisfying any {\em a priori\/} irreducibility condition.


\section{Adjoint polynomials}\label{partIII}

Our main motivation in establishing Theorems~\ref{thm:R} and~\ref{thm:P} is the general study
of Segre classes: constraints on the possible numerators of Segre zeta functions translate
into constraints on what classes can be Segre classes of subschemes of e.g., projective space,
thus may be an aid in their computation. In this section we provide an alternative motivation 
for this work, by interpreting Theorem~\ref{thm:R} in the special case of {\em monomial\/} 
ideals. There is a connection between Segre zeta functions of monomial ideals and {\em adjoint
polynomials,\/} first noted by Kathl\'en Kohn and Kristian Ranestad 
(\cite[Proposition~1]{MR4156995}). Kohn and Ranestad focused on the numerator of $\zeta_I$.
We recover an analogous result for the numerator of $1-\zeta_I$, with the advantage that the
corresponding polytopal object is convex. In short, we will prove that adjoint polynomials of certain
convex polytopes are necessarily covolume polynomials; in particular, they are M-convex and
sectionally log-concave.

Following Joe Warren (\cite[\S2]{MR1431788}), we consider {\em polyhedral cones,\/} that is, 
convex hulls of finite sets of rays emanating from the origin in an $\Rbb$-vector space $V$.
For a set of nonzero vectors $S\in V$, we will denote by $\cP_S$ the convex polyhedral cone
obtained by taking the convex hull of the rays through the vectors $\uv\in S$.

For example, one may consider $\cP_S$ for $S$ the set of vertices of a convex polytope
embedded in a hyperplane $x_0=1$. It is convenient to think of polyhedral cones as a 
`projective' version of polytopes.

We will denote by $V(\cP)$ a set of vectors spanning the vertex rays of $\cP$. Each such
vector is determined up to the choice of a positive real scalar. We may take $V(\cP_S)$
to be a subset of $S$.

A {\em triangulation\/} of a polyhedral cone $\cP$ of dimension $d$ is a collection of 
$d$-dimensional simplicial cones whose union is $\cP$, whose vertex rays are subsets of the
rays of $\cP$ and such that the intersections of any two simplicial cones are faces of both.

\begin{definthm}\label{def:adjpol}
(\cite{MR1431788})
Let $\cP$ be a polyhedral cone in $\Rbb^{\ell+1}$ and let $T(\cP)$~be a triangulation
of~$\cP$. We define the {\em adjoint polynomial\/} of $\cP$ 
(with respect to the choice~$V(\cP)$ of vertex vectors) to be
\begin{equation}\label{eq:adjpol}
\Abb_\cP(t_0,\dots, t_\ell) = \sum_{\sigma\in T(\cP)} \Vol(\sigma) 
\prod_{(v_0,\dots, v_\ell)\in V(\cP) \smallsetminus V(\sigma)}(v_0 t_0+\cdots +v_\ell t_\ell)\saf.
\end{equation}
This definition is independent of the chosen triangulation.
\qede\end{definthm}

The quantity $\Vol(\sigma)$ in~\eqref{eq:adjpol} is the absolute value of the determinant of 
the matrix whose entries are the coordinates of the vertices of the simplicial cone $\sigma$. 
It should be viewed as the volume of $\sigma$, up to a normalization factor.

The adjoint polynomial in Definition~\ref{def:adjpol} is determined by $\cP$ up to a positive 
real scalar factor: replacing a vertex $\uv$ of $\cP$ by a multiple $\lambda\uv$ with 
$\lambda\in \Rbb_{>0}$ has the effect of multiplying each summand of $\Abb_\cP(\ut)$ by 
$\lambda$, since this either multiplies by $\lambda$ a column of the determinant computing 
$\Vol(\sigma)$ or exactly one of the other factors. We could fix this factor, for example by requiring 
all vertices to have length~$1$, or by requiring the non-coordinate vertices to lie on the hyperplane 
$\{a_0=1\}$, but no such choice is necessary for what follows.

The independence of the definition of $\Abb_\cP(\ut)$ on the choice of a triangulation is
proved in~\cite[Theorem~4]{MR1431788}; also see~Remark~\ref{rem:inde}.

The adjoint polynomial $\Abb_\cP(\ut)$ of a polyhedral cone $\cP$ in $\Rbb^{\ell+1}$ is a 
homogeneous polynomial of degree $|V(\cP)|-\ell-1$
(\cite[Theorem~1]{MR1431788}). 
Note that the coefficients of an adjoint polynomial are not necessarily nonnegative.

\begin{example}
The adjoint polynomial for the cone $\cP_S\subseteq \Rbb^3$ for 
$S=\{(1,0,0), (1,-1,0),$ $(1,0,-1), (1,-1,-1)\}$ (the vertices of a square in the hyperplane
$x_0=1$) is $\Abb_{\cP_S}(\ut)=2t_0-t_1-t_2$.
\qede\end{example}

As we are interested in studying Lorentzian properties of adjoint polynomials, it is natural
to impose a condition guaranteeing that the coefficients are nonnegative. The most
natural such condition is that the spanning set $S$ should be contained in the nonnegative orthant. 
We will further require that the cone should contain all but one coordinate rays, that is, that 
the cone shares a face with the nonnegative orthant.
A mild generalization of this condition will be considered in Corollary~\ref{cor:orthan}.

\begin{theorem}\label{thm:adjpol}
Let $S\subseteq \Rbb_{\ge 0}^{\ell+1}$ be a finite set including the coordinate vectors $\ue_1,
\dots, \ue_\ell$. Then $\Abb_{\cP_S}(t_0,\dots, t_\ell)$ is a covolume polynomial.
In particular, $\Abb_{\cP_S}$ is M-convex and sectional log-concave, and 
for all integers $n_0,\dots, n_\ell\ge 0$ such that
\begin{equation}\label{eq:flip}
u_0^{n_0}\cdots u_\ell^{n_\ell} \Abb_{\cP_S} \left(\frac 1{u_0},\dots, \frac 1{u_\ell}\right)
\end{equation}
is a polynomial, the normalization of this polynomial is Lorentzian.
\end{theorem}

\begin{remark}
Thus, the adjoint polynomials considered in Theorem~\ref{thm:adjpol} are {\em dually 
Lorentzian\/} in the sense of~\cite{dualLor}.
\qede\end{remark}

\begin{example}
Let $\cP$ be the polyhedral cone with vertex rays spanned by the vectors
\[
\uv_1:=(1,1,\sqrt 2,0,0),\quad \uv_2:=(1,\sqrt 2,1,0,0),\quad 
\uv_3:=(1,0,0,1,0),\quad \uv_4:=(1,0,0,0,1)
\]
as well as the coordinate directions
\[
\ue_1=(0,1,0,0,0),\quad
\ue_2=(0,1,0,0,0),\quad
\ue_3=(0,1,0,0,0),\quad
\ue_4=(0,1,0,0,0)\saf.
\]
A triangulation of $\cP$ consists of the four simplices
\[
\langle \uv_1,\uv_2,\uv_3,\uv_4,\ue_1\rangle,\quad
\langle \uv_1,\uv_3,\uv_4,\ue_1,\ue_2\rangle,\quad
\langle \uv_1,\uv_2,\ue_2,\ue_1,\ue_4\rangle,\quad
\langle \uv_3,\ue_1,\ue_2,\ue_3,\ue_4\rangle\saf,
\]
and the adjoint polynomial of $\cP$ is
\begin{multline*}
\Abb_\cP(t_0,\dots, t_4)=t_0^3 + (1 +  \sqrt 2) t_ 1 t_0^2 + (1 +  \sqrt 2) t_2 t_0^2 + t_3 t_0^2 + t_4 t_0^2 
+  \sqrt 2 t_ 1^2 t_0 + 3 t_0 t_ 1 t_2\\ 
 + (1 +  \sqrt 2) t_ 1 t_3 t_0 + (1 +  \sqrt 2) t_4 t_ 1 t_0 +  \sqrt 2 t_2^2 t_0 + (1 +  \sqrt 2) t_2 t_3 t_0 
 + (1 +  \sqrt 2) t_4 t_2 t_0 + t_3 t_4 t_0 \\ 
+  \sqrt 2 t_1^2 t_3+  \sqrt 2 t_1^2 t_4 + 3 t_1 t_2 t_3 + 3 t_1 t_2 t_4 +  \sqrt 2 t_1 t_3 t_4 
+  \sqrt 2 t_2^2 t_3 +  \sqrt 2 t_2^2 t_4 + t_2 t_3 t_4  \sqrt 2\saf.
\end{multline*}
The 21 terms in its support (of 35 in the simplex of tuples $(a_0,\dots, a_4)$ of nonnegative 
integers with $\sum a_i=3$) form an M-convex set, as prescribed by Theorem~\ref{thm:adjpol}.
In fact, the normalization of $\Abb_\cP$ is Lorentzian (the polynomial itself is not Lorentzian).
The reader may verify that the normalization of 
\[
u_0^3 u_1^2 u_2^2 u_3 u_4\cdot \Abb_\cP\left(\frac 1{u_0},\frac 1{u_1},\frac 1{u_2},\frac 1{u_3},
\frac 1{u_4}\right)\saf,
\]
a homogeneous degree-$6$ polynomial, is Lorentzian as stated in Theorem~\ref{thm:adjpol}.
\qede\end{example}

Theorem~\ref{thm:adjpol} is proved by relating the adjoint polynomial to a Segre zeta function.
As mentioned at the beginning of this section, Kathl\'en Kohn and Kristian Ranestad express
such a relation in~\cite[Proposition~1]{MR4156995}; their result deals with $\zeta_I$, and
correspondingly with not necessarily convex regions. We adopt the context of convex
polyhedral cones and focus on the function $1-\zeta_I$.
For these considerations, $I$ is a set of {\em monomials.\/}

Precisely, let $F$ be a finite set of vectors $(v_1,\dots, v_\ell)\in \Zbb_{\ge 0}^\ell$.
We can associate with $F$ two objects:
\begin{itemize}
\item 
The convex polyhedral cone $\cP_S$ in $\Rbb^{\ell+1}$ spanned by the set 
\begin{equation}\label{eq:monS}
S=\{\uv:=(1,v_1,\dots, v_\ell)\,|\, (v_1,\dots, v_\ell)\in F\}\cup \{\ue_1,\dots, \ue_\ell\} \saf.
\end{equation}
\item 
The ideal $k[x_1,\dots, x_\ell]$ generated by the set $I$ of {\em monomials\/}
\[
x_1^{v_1}\cdots x_\ell^{v_\ell}\quad,\quad \text{with $(v_1,\dots, v_\ell)\in F$}
\]
and the corresponding Segre zeta function $\zeta_I(t_1,\dots, t_\ell)$.
\end{itemize}
Here we are viewing each monomial as (trivially) homogenous in each variable, hence as a 
multihomogeneous polynomial.

\begin{prop}\label{prop:adjasS}
With notation as above, let $\uv_0,\dots, \uv_r\in V(\cP_S)$ be the vertex ray vectors
other than $\ue_1,\dots, \ue_\ell$. Then
\[
1-\zeta_I(t_1,\dots, t_\ell)=\left.\frac{\Abb_{\cP_S}(t_0,\dots, t_\ell)}
{\prod_{i=0}^r \uv_i\cdot (t_0,\dots, t_\ell)}
\right|_{t_0=1}\saf.
\]
\end{prop}

\begin{remark}\label{rem:adjasSR}
In other words, the adjoint polynomial $\Abb_{\cP_S}$ is the degree-$r$ homogenization
of the polynomial $R$ appearing in Theorem~\ref{thm:R}.
\qede\end{remark}

\begin{proof}
The Segre class $s(Z,W)$ of a monomial subscheme of a variety $W$ may be
expressed as an integral. Explicitly, let $X_1,\dots, X_\ell$ be hypersurfaces of a variety $W$ 
meeting with normal crossings; let $Z$ be the closed subscheme cut out by monomials 
in the $X_i$ with classes $v_1 X_1+\cdots+v_\ell X_\ell$ as $(v_1,\dots, v_\ell)\in \Zbb_{\ge 0}^\ell$ 
ranges over a finite set $F$; and let $N\subseteq \Rbb^\ell$ be the complement in the nonnegative
orthant $\Rbb_{\ge 0}^\ell$ of the convex hull of the translates of the nonnegative orthant by the points 
in $F$. Then, using coordinates $(a_1,\dots, a_\ell)$ for $\Rbb^\ell$, we have
(\cite[Theorem~1.1]{MR3576538})
 \[
s(Z,W)=
\int_N \frac{\ell!\, X_1\cdots X_\ell\, da_1\cdots da_\ell}
{(1+X_1 a_1+\cdots +X_\ell a_\ell)^{\ell+1}}\saf.
\]
(The result in \cite{MR3576538} holds in a somewhat more general setting, but this won't play
a role here.)
This should be interpreted by computing the integral on the right-hand side as a rational
function in the parameters $X_j$, then expanding this function as a power series, and
evaluating each monomial $X_1^{a_1}\cdots X_\ell^{a_\ell}$ as a class in $A_*(Z)$;
see~\cite{MR3576538} for a more extensive discussion. For instance, monomials 
$X_1^{a_1}\cdots X_\ell^{a_\ell}$ with $\sum_j a_j>\dim W$ evaluate to $0$, so that the
integral may be written as a sum of only finitely many terms for any given $Z$ and~$W$.

If $N$ is the nonnegative orthant itself, the integral evaluates to $1$. Therefore, 
\begin{equation}\label{eq:intMS}
[W]-\iota_* s(Z,W)=
\int_{\Rbb_{\ge 0}^\ell \smallsetminus N} \frac{\ell!\, X_1\cdots X_\ell\, da_1\cdots da_\ell}
{(1+X_1 a_1+\cdots +X_\ell a_\ell)^{\ell+1}}\saf,
\end{equation}
where $\iota: Z\to W$ is the inclusion. Adopting the notation introduced in~\eqref{eq:monS},
viewing the cone $\cP_S$ in $\Rbb^{\ell+1}$, with coordinates $(a_0,\dots, a_\ell)$,
and identifying the hyperplane $\{a_0=1\}$ with $\Rbb^\ell$, we have
\[
\Rbb_{\ge 0}^\ell \smallsetminus N = \{a_0=1\}\cap \cP_S\saf.
\]
Applying~\eqref{eq:intMS} with $W=\Pbb^{n_1}\times\cdots\times \Pbb^{n_\ell}$, $n_j\gg 0$, 
and $X_j=$ the hypersurface in $W$ obtained by restricting the $j$-th factor to a hyperplane,
we get
\begin{equation}\label{eq:intS}
1-\zeta_I(t_1,\dots, t_\ell)=
\int_{\{a_0=1\}\cap \cP_S} \frac{\ell!\, t_1\cdots t_\ell\, da_1\cdots da_\ell}
{(1+t_1 a_1+\cdots +t_\ell a_\ell)^{\ell+1}}\saf.
\end{equation}
Now let $T(\cP_S)$ be a triangulation of $\cP_S$, as in Definition~\ref{def:adjpol}.
According to~\cite[\S2]{MR3576538}, a simplicial cone with vertices $\uv_0,\dots, \uv_d$ and 
$\ue_{i_1},\dots, \ue_{i_{\ell-d}}$, where $\uv_j$ correspond to vectors in~$F$, contributes
{\small
\[
\frac{\Vol(\sigma)\, t_1\cdots t_\ell}
{\prod_{j=0}^d (1+\uv_{j1}t_1+\cdots +\uv_{j\ell}t_\ell)\prod_{j=1}^{\ell-d}t_j}
=\left.\frac{\Vol(\sigma)\, t_1\cdots t_\ell}
{\prod_{j=0}^d (\uv_j\cdot \ut)\prod_{j=1}^{\ell-d} (\ue_j\cdot \ut)}
\right|_{t_0=1} 
=\left.\frac{\Vol(\sigma)\, t_1\cdots t_\ell}
{\prod_{\uv\in V(\sigma)}(\uv\cdot \ut)}
\right|_{t_0=1}
\]}
to the integral in~\eqref{eq:intS}, and therefore
\begin{equation}\label{eq:intap}
\begin{aligned}
1-\zeta_I(t_1,\dots, t_\ell)&=
\int_{\{a_0=1\}\cap \cP_S} \frac{\ell!\, t_1\cdots t_\ell\, da_1\cdots da_\ell}{(1+t_1 a_1+\cdots +t_\ell a_\ell)^{\ell+1}} \\
&=\left.\frac{\sum_{\sigma\in T(\cP_S)}\Vol(\sigma)\,\prod_{\uv\in V(\cP_S)\smallsetminus V(\sigma)}
(\uv \cdot \ut)\, t_1\cdots t_\ell}
{\prod_{\uv\in V(\cP_S)}(\uv \cdot \ut)}
\right|_{t_0=1} \\
&=\left.\frac{\Abb_{\cP_S}(t_0,\dots, t_\ell)\, t_1\cdots t_\ell}
{\prod_{\uv\in V(\cP_S)}(\uv \cdot \ut)}
\right|_{t_0=1} \\
&=\left.
\frac{\Abb_{\cP_S}(t_0,\dots, t_\ell)}
{\prod_{i=0}^r (\uv_i\cdot \ut)}
\right|_{t_0=1}\saf,
\end{aligned}
\end{equation}
as stated.
\end{proof}

\begin{remark}\label{rem:inde}
As an aside, the statement of Proposition~\ref{prop:adjasS} implies that the definition of adjoint 
polynomial is independent of the triangulation (cf.~\cite[Theorem~4]{MR1431788}), for the 
convex polyhedral cones considered here. In fact, the integral expression for the adjoint
polynomial worked out in~\eqref{eq:intap} extends to arbitrary polyhedral cones (but not, to
our knowledge, the interpretation in terms of a Segre zeta function), and this implies the
independence on the choice of triangulation in general, providing an alternative to
\cite[Theorem~4]{MR1431788}. 
\qede\end{remark}

\begin{proof}[Proof of Theorem~\ref{thm:adjpol}]
The adjoint polynomial of a polyhedral cone depends continuously on the coordinates of 
vectors spanning its vertex rays, so we may express it as a limit of adjoint polynomials of 
polyhedral cones whose vertex rays contain vectors with rational components. 
By definitions, limits of covolume polynomials are covolume polynomials, so we are
reduced to the case of polyhedral cones with rational vertex rays. 
In fact, by choosing a large enough integer $d$, we may assume that all non-coordinate 
vectors in $S$ are of the form $(d,v_1,\dots, v_\ell)\in \Zbb_{\ge 0}^{\ell+1}$. We are then 
reduced to proving the assertion of Theorem~\ref{thm:adjpol} for polyhedral cones 
spanned by a set $S$ consisting of $\ue_1,\dots, \ue_\ell$ and of vectors of this type.

For this, let
\[
S'=\{\ue_1,\dots, \ue_\ell\} \cup \{(1,v_1,\dots, v_\ell)\,|\, (d,v_1,\dots, v_\ell)\in S\}\saf.
\]
The set $S'$ is of the form considered in Proposition~\ref{prop:adjasS}. It follows that
$\Abb_{\cP_{S'}}$ is the homogenization of the numerator of $1-\zeta_I$ for a suitable
set $I$ of monomials, cf.~Remark~\ref{rem:adjasSR}. By Theorem~\ref{thm:R},
$\Abb_{\cP_{S'}}$ is a covolume polynomial. Now Definition~\ref{def:adjpol} implies that
\[
\Abb_{\cP_S}(t_0,t_1,\dots, t_\ell) = d\, \Abb_{\cP_{S'}}(dt_0,t_1,\dots, t_\ell)\saf;
\]
therefore, $\Abb_{\cP_S}$ is obtained from a covolume polynomial by a nonnegative
change of coordinates. By Theorem~\ref{thm:nonnegcc} we can conclude that
$\Abb_{\cP_S}$ is a covolume polynomial, as needed.
\end{proof}

We do not know whether the adjoint polynomials considered in Theorem~\ref{thm:adjpol}
are also necessarily Lorentzian after normalization. An affirmative answer to 
Question~\ref{que:norLor} would imply that this is the case.

We also do not know whether Theorem~\ref{thm:adjpol} extends to all convex polyhedral
cones contained in the nonnegative orthant. 
The hypothesis of convexity cannot be removed, in the sense that there are non-convex unions of
polyhedral cones contained in the nonnegative orthant and whose adjoint polynomial is 
not M-convex, hence not a covolume polynomial. 

\begin{example}
Let $\uv_1=(1,1,1,0)$, $\uv_2=(1,1,0,1)$, along with the coordinate vectors $\ue_0=(1,0,0,0),\dots, 
\ue_3=(0,0,0,1)$ in $\Rbb^4$. The adjoint polynomial of the union of the three simplicial 
cones with vertex rays $\ue_0\ue_1\uv_1\uv_2$, $\ue_0\uv_1\uv_2\ue_3$, 
$\ue_0\uv_1\ue_2\ue_3$ is
\[
t_0 t_1+ t_1^2+t_1t_2+t_1t_3+t_2t_3
\]
as the reader may verify. 
This polynomial is not M-convex or sectional log-concave.
(This example is the polyhedral version of Example~\ref{ex:Pnotcov}; 
see~\cite[Proposition~1]{MR4156995}.)
\qede\end{example}

On the other hand, Theorem~\ref{thm:adjpol} extends easily to the following class of
polyhedral cones.

\begin{defin}\label{def:orthantal}
We say that a convex polyhedral cone is {\em orthantal\/} if it shares a face with a 
full dimensional simplicial cone enclosing it and contained in the nonnegative orthant.
\qede\end{defin}

The intersection of an orthantal polyhedral cone with the hyperplane $a_0=1$ will be a 
convex polytope enclosed in a simplex with which it shares a face, and we are assuming 
that this simplex is contained in the nonnegative orthant. 
\begin{center}
\includegraphics[scale=.4]{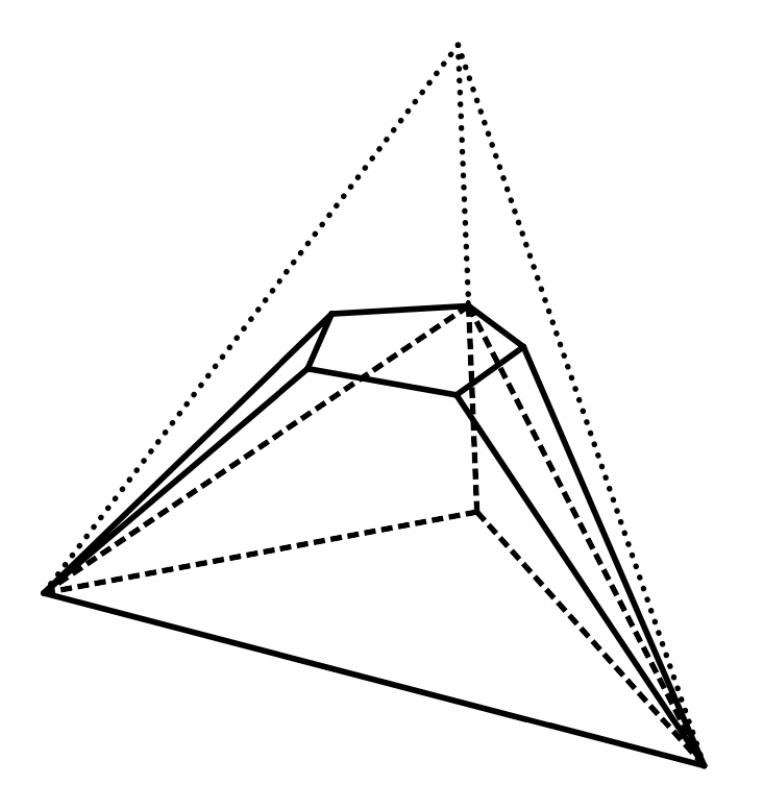}
\end{center}
The polyhedral cones considered in Theorem~\ref{thm:adjpol} are orthantal, with the 
simplicial cone equal to the nonnegative orthant itself.

\begin{corol}\label{cor:orthan}
The adjoint polynomial of a convex orthantal polyhedral cone is a covolume polynomial.
Therefore, it is M-convex, sectional log-concave, and the normalizations of the corresponding
polynomials~\eqref{eq:flip} are Lorentzian (that is, adjoint polynomials of convex orthantal
polyhedral cones are dually Lorentzian).
\end{corol}

We need the following elementary fact recording the effect of a linear transformation on an 
adjoint polynomial.
Let $A=(a_{ij})$, $0\le i,j\le \ell$ be a nonsingular matrix and let $\cP$ be any polyhedral cone 
in $\Rbb^{\ell+1}$. Denote by $A\cP$ the polyhedral cone with vertices $A\cdot \uv$ for 
$\uv\in V(\cP)$. Denote by $A^{\mathrm{t}}$ the transpose of $A$.

\begin{lemma}\label{lem:adjlt}
\[
\Abb_{A\cP}(\ut)=|\det(A)|\cdot \Abb_{\cP}\left(A^\mathrm{t} (\ut)\right)\saf.
\]
\end{lemma}

\begin{remark}
Since the adjoint polynomial is only defined up to a positive factor, the term $|\det(A)|$
in Lemma~\ref{lem:adjlt} is actually superfluous.
\qede\end{remark}

\begin{proof}
By definition,
\[
\Abb_{A\cP}(\underline t) = \sum_{\sigma\in T(A\cP)} \Vol(\sigma) 
\prod_{(w_0,\dots, w_\ell)\in V(A\cP) \smallsetminus V(\sigma)}(w_0 t_0+\cdots +w_\ell t_\ell)
\]
where $T(A\cP)$ denotes a triangulation of $A\cP$. A triangulation of $A\cP$ can be 
obtained by mapping by $A$ the simplices of a triangulation for $\cP$; the volume
of the simplices in the triangulation is multiplied by $\det(A)$. Therefore
{\small
\begin{align*}
\Abb_{A\cP}&(\underline X) = \sum_{\sigma\in T(\cP)} |\det(A)|\Vol(\sigma) 
\prod_{\uv=(v_0,\dots, v_\ell)\in V(\cP) \smallsetminus V(\sigma)}
((A\uv)_0 t_0+\cdots +(A\uv)_\ell t_\ell) \\
&= |\det(A)|\sum_{\sigma\in T(\cP)} \Vol(\sigma) 
\prod_{(v_0,\dots, v_\ell)\in V(\cP) \smallsetminus V(\sigma)}
\left(\sum_{i=0}^\ell\sum_{j=0}^\ell a_{ij} v_j t_i\right) \\
&= |\det(A)|\sum_{\sigma\in T(\cP)} \Vol(\sigma) 
\prod_{(v_0,\dots, v_\ell)\in V(\cP) \smallsetminus V(\sigma)}
\left(v_0\sum_{i=0}^\ell a_{i0} t_i +\cdots+ v_\ell \sum_{i=0}^\ell a_{in} t_i\right)
\end{align*}}
with the stated consequence.
\end{proof}

\begin{proof}[Proof of Corollary~\ref{cor:orthan}]
Let $\cP$ be an orthantal polyhedral cone. 
Let $(a_{0j},\dots, a_{\ell j})$, $j=0,\dots,\ell$, be the vectors in $V(\Sigma)$ for the simplex
$\Sigma$ containing $\cP$ and contained in $\Rbb_{\ge 0}^{\ell+1}$,
with $(a_{0j},\dots, a_{\ell j})$, $j=1,\dots,\ell$ the vertices of the simplicial face of $\cP$ 
in common with $\Sigma$. Let 
\[
A=\begin{pmatrix}
a_{00} & a_{01} & \cdots & a_{0\ell} \\
a_{10} & a_{11} & \cdots & a_{1\ell} \\
\vdots & \vdots & \ddots & \vdots \\
a_{\ell 0} & a_{\ell 1} & \cdots & a_{\ell \ell}
\end{pmatrix}\saf.
\]
Since $\Sigma$ is a full dimensional simplex contained in $\Rbb_{\ge 0}^{\ell+1}$,
$A$ has nonnegative entries and is nonsingular. 
By Lemma~\ref{lem:adjlt},
\[
\Abb_{\cP}(\ut)=\Abb_{A^{-1}\cP}\left(\sum_{i=0}^\ell a_{i0}t_i,\dots, \sum_{i=0}^\ell a_{i\ell}t_i\right)\saf.
\]
The chosen simplicial face of the polyhedral cone $A^{-1}\cP$ has vertices along $\ue_1,\dots,
\ue_\ell$, while the other vertex of the enclosing simplex is mapped to $\ue_0$. Therefore 
$A^{-1}\cP$ is enclosed in the simplex with vertices $\ue_0,\dots, \ue_\ell$, that is, the 
nonnegative orthant.

By Theorem~\ref{thm:adjpol}, $\Abb_{A^{-1}\cP}$ is a covolume polynomial. Since $\Abb_{\cP}$
is obtained from $\Abb_{A^{-1}\cP}$ by a nonnegative change of variables, 
Theorem~\ref{thm:nonnegcc} implies that $\Abb_{\cP}$ is also a covolume polynomial, 
completing the proof.
\end{proof}

\newcommand{\etalchar}[1]{$^{#1}$}

\end{document}